\documentclass{memo-l}

\usepackage{amsfonts}
\usepackage{amsmath}
\usepackage{amsthm}
\usepackage{enumerate}
\usepackage{amssymb}
\usepackage{enumitem}
\usepackage{bbm}
\usepackage{amscd}
\usepackage[all]{xy}
\usepackage{hyperref}
\usepackage{aliascnt}
\usepackage{verbatim}
\usepackage[sort&compress,numbers]{natbib}
\usepackage[nameinlink]{cleveref}

\newcommand{\ep}{\varepsilon}

\newcommand{\Z}{{\mathbb{Z}}}

\newcommand{\C}{{\mathbb{C}}}
\newcommand{\N}{{\mathbb{N}}}

\newcommand{\Hi}{{\mathcal{H}}}
\newcommand{\K}{{\mathcal{K}}}

\newcommand{\B}{{\mathcal{B}}}
\newcommand{\U}{{\mathcal{U}}}

\newcommand{\Ot}{{\mathcal{O}_2}}
\newcommand{\OI}{{\mathcal{O}_{\infty}}}

\newcommand{\beq}{\begin{equation}}
\newcommand{\eeq}{\end{equation}}
\newcommand{\beqa}{\begin{eqnarray*}}
\newcommand{\eeqa}{\end{eqnarray*}}
\newcommand{\bal}{\begin{align*}}
\newcommand{\eal}{\end{align*}}
\newcommand{\bi}{\begin{itemize}}
\newcommand{\ei}{\end{itemize}}
\newcommand{\be}{\begin{enumerate}}
\newcommand{\ee}{\end{enumerate}}

\theoremstyle{plain}
\newtheorem{lma}{Lemma}[section]

\newaliascnt{thmCt}{lma}
\newtheorem{thm}[thmCt]{Theorem}
\aliascntresetthe{thmCt}

\newaliascnt{corCt}{lma}
\newtheorem{cor}[corCt]{Corollary}
\aliascntresetthe{corCt}

\newaliascnt{propCt}{lma}
\newtheorem{prop}[propCt]{Proposition}
\aliascntresetthe{propCt}

\newtheorem*{thm*}{Theorem}
\newtheorem*{cor*}{Corollary}
\newtheorem*{prop*}{Proposition}
\newtheorem*{qst*}{Question}
\newtheorem*{pbm*}{Problem}

\newcounter{theoremintro}

\theoremstyle{definition}

\newaliascnt{pgrCt}{lma}

\aliascntresetthe{pgrCt}

\newaliascnt{dfCt}{lma}
\newtheorem{df}[dfCt]{Definition}
\aliascntresetthe{dfCt}

\newaliascnt{remCt}{lma}
\newtheorem{rem}[remCt]{Remark}
\aliascntresetthe{remCt}

\newaliascnt{remsCt}{lma}

\aliascntresetthe{remsCt}

\newaliascnt{exCt}{lma}
\newtheorem{ex}[exCt]{Exercise}
\aliascntresetthe{exCt}

\newaliascnt{egCt}{lma}
\newtheorem{eg}[egCt]{Example}
\aliascntresetthe{egCt}

\newaliascnt{egsCt}{lma}
\newtheorem{egs}[egsCt]{Examples}
\aliascntresetthe{egsCt}

\newaliascnt{qstCt}{lma}
\newtheorem{qst}[qstCt]{Question}
\aliascntresetthe{qstCt}

\newaliascnt{cnjCt}{lma}
\newtheorem{cnj}[cnjCt]{Conjecture}
\aliascntresetthe{cnjCt}

\newaliascnt{pbmCt}{lma}
\newtheorem{pbm}[pbmCt]{Problem}
\aliascntresetthe{pbmCt}

\newaliascnt{notaCt}{lma}
\newtheorem{nota}[notaCt]{Notation}
\aliascntresetthe{notaCt}

\theoremstyle{theorem}

\AtBeginDocument{   \def\MR#1{}
}

\input{tcilatex}

\numberwithin{section}{chapter}

\newcommand{\tfae}{the following are equivalent}

\newcommand{\ca}{C$^*$-algebra}

\newcommand{\uca}{unital C*-algebra}

\newcommand{\id}{{\mathrm{id}}}

\newcommand{\supp}{{\mathrm{supp}}}

\newcommand{\Aut}{{\mathrm{Aut}}}

\newcommand{\Hom}{{\mathrm{Hom}}}
\newcommand{\Ext}{{\mathrm{Ext}}}

\begin{document}

\title[]{Classifiability of crossed products}


\author{Eusebio Gardella}
\address{Department of Mathematical Sciences, Chalmers University of
Technology and University of Gothenburg, Gothenburg SE-412 96, Sweden.}
\email{gardella@chalmers.se}
\urladdr{www.math.chalmers.se/~gardella}

\date{\today}

\thanks{The author was partially supported by a starting grant of 
the Swedish Research Council and by the Royal Society of Arts and Sciences in Gothenburg.}

\begin{abstract}
This survey studies C*-algebraic crossed products arising from topological dynamical systems with an eye toward their classifiability in the sense of the Elliott program. We introduce the crossed product construction for actions by discrete groups in full generality, and then focus on commutative systems to establish some of the fundamental structural results: nuclearity of the crossed product is equivalent to amenability of the action, and in this setting simplicity is equivalent to the combination of minimality and topological freeness. With these properties in place, the only condition left to translate into dynamical terms is tensorial absorption of the Jiang-Su algebra $\mathcal{Z}$. We present the main dynamical tools known to obtain $\mathcal{Z}$-stability, both in the settings of amenable and nonamenable groups, and highlight the two conjectures that are believed to capture the full picture. The aim of the survey is to provide both an accessible entry point and a comprehensive reference on the classification of crossed products, describing the state of the art and a number of central open problems in the field.
\end{abstract}

\maketitle

\tableofcontents

\mainmatter

\renewcommand*{\thetheoremintro}{\Alph{theoremintro}}

\renewcommand*{\thetheoremintro}{\Alph{theoremintro}}

\chapter{Introduction}

The study of C*-algebraic crossed products arising from topological dynamical systems lies at the heart of the interplay between operator algebras, topology, and dynamics. Crossed products provide a construction by which an action of a group on a topological space gives rise to a new C*-algebra, encoding dynamical and algebraic information simultaneously. Indeed, the structure of the crossed product reflects a great deal of information about the orbit complexity, the presence (or absence) of invariant measures, and finer regularity properties of the action. Over the years, crossed products have played a dual role: they provide a vast supply of natural examples, among others for the classification program, and they provide geometric and/or dynamical coordinates through which known C*-algebras can be studied; this is most clear in the case of the irrational rotation algebras.

This survey is organized with two complementary goals. First, it is meant to provide a self–contained introduction to crossed products of actions by discrete groups on C*-algebras, which we do in Chapter~\ref{ch:CP}. Assuming only basic knowledge in C*-algebras, we construct full and reduced crossed products. We begin with covariant representations and integrated representations, defining full and reduced crossed products and discussing 
a relevant case in which they agree, namely when the group is amenable. 
Additionally, we collect the fundamental properties that are used throughout, including the description of the \emph{full} crossed product as a universal C*-algebra on generators and relations (\autoref{thm:FullCPDiscreteUniv}) and the fact that the canonical conditional expectation on the \emph{reduced} crossed product is faithful (\autoref{thm:CondExpFaithful}). Along the way we present concrete examples of dynamical systems and compute their crossed products. 

The second focus is on \emph{classifiability}: when does a crossed product fall within the class of C*-algebras covered by Elliott's classification program? 
The Elliott program has led to a definitive classification result for a robust class of simple, nuclear C*-algebras by K-theoretic and tracial data. For the purposes of this survey, it is convenient to phrase it as follows.

\begin{thm*}[Classification]
Simple, separable, unital, nuclear, $\mathcal{Z}$-stable C*-algebras satisfying the Universal Coefficient Theorem (UCT) are completely classified by their K-theories and traces (known as the Elliott invariant).
\end{thm*}

We will call a C*-algebra \emph{classifiable} if it satisfies the assumptions of the above theorem.
Kirchberg's dichotomy asserts that a classifiable C*-algebra is either stably finite or purely infinite, and the proof of the classification theorem splits accordingly into two disjoint cases: the purely infinite case was settled by Kirchberg and Phillips almost thirty years ago \cite{Phi_classification_2000}, while the stably finite case has only been resolved in the past decade through the combination of several deep advances \cite{EllGonLinNiu_classification_2015, TikWhiWin_quasidiagonality_2017, CasEviTikWhiWin_nuclear_2021}. Winter's ICM address \cite{Win_structure_2018} provides an excellent overview of the development of the subject and the theory underlying these results. Additionally, \autoref{app:Elliott} gives a historical account and motivation of the classification program. 

With such a powerful theorem at hand, a natural and urgent task is to identify broad and interesting families of C*-algebras that satisfy its assumptions. Crossed products arising from topological dynamical systems are among the most important such families: given an action $G \curvearrowright X$ of a countable group on a compact metrizable space, the crossed product $C(X) \rtimes G$ encodes the dynamics of the system in operator-algebraic form. 
For researchers primarily interested in topological dynamics, the problem of determining when a crossed product is classifiable ought to be answered in dynamical terms, that is, we are seeking for a list of (intrinsic) properties of a dynamical system $G\curvearrowright X$ that are equivalent to classifiability of $C(X) \rtimes G$. Kirchberg’s dichotomy takes a particularly nice form
in this setting, since a simple, nuclear crossed product $C(X)\rtimes G$ is stably finite if and only if G is amenable; see \autoref{prop:GamenIFFinvmeas}.

Of the hypotheses in the classification theorem, both unitality and separability of $C(X)\rtimes G$ will be automatic for the actions we consider\footnote{Explicitly, unitality of $C(X)\rtimes G$ is equivalent to the combination of compactness of $X$ and discreteness of $G$, while separability is equivalent to the combination of metrizability of $X$ and countability of $G$.}. On the other hand, we will see that nuclearity of $C(X)\rtimes G$ is equivalent to a property of $G\curvearrowright X$ introduced by Anantharaman-Delaroche, called amenability (see \autoref{df:AmenAc} and \autoref{thm:AmenIffNuc}). In this case, $C(X) \rtimes G$ also satisfies the UCT by \cite[Theorem~10.9]{Tu_conjecture_1999}. Moreover, assuming amenability of $G\curvearrowright X$, simplicity of the crossed product is equivalent to the combination of minimality and topological freeness (see \autoref{df:EssFree} and \autoref{thm:SimpleCP}). 
These results are not treated as black boxes, and proofs are given in sufficient detail to make the survey self-contained for readers with basic knowledge in C*-algebras. We also discuss traces on crossed products and their relation to 
invariant probability measures on the space; see \autoref{thm:Ursu}.
It follows that amenability, minimality, and topological freeness are \emph{necessary} dynamical conditions for classifiability, and the remaining challenge is to determine when $C(X) \rtimes G$ is $\mathcal{Z}$-stable.

Although we do not assume familiarity with the Jiang-Su algebra $\mathcal{Z}$, we 
do not review its construction in these notes; this is not strictly necessary either, since we are interested in the property of absorbing $\mathcal{Z}$ tensorially. Said absorption condition can be characterized (in the simple, nuclear case) without referring to $\mathcal{Z}$ and only using matrix algebras (see \autoref{thm:ZstTrZst}) and in practice we take this criterion as the definition of $\mathcal{Z}$-stability. After specializing this condition to crossed products, we 
move on to the topic of active research in this setting: relating $\mathcal{Z}$-stability with dynamical properties of the system -- this task is not finished, and there is an increasing amount of research effort put into completing it.

This ``open end" of the subject, and the organizing principle for the latter part of this survey, can be framed in terms of two conjectures: one covering (free) actions of \emph{amenable} groups, and another one for actions of \emph{nonamenable} groups. We state them here in a brief form to guide the reader, who is referred to later chapters for motivation and definitions, as well as for known partial results and limiting examples.

\begin{cnj}[Amenable groups]
	Let $G\curvearrowright X$ be a free, minimal, action of an amenable group. Then the following are equivalent:
	\be\item The crossed product $C(X)\rtimes G$ is $\mathcal{Z}$-stable;
	\item $G\curvearrowright X$ has the small boundary property\footnote{This means that there is a basis $\mathcal{B}$ for the topology of $X$ such that $\mu(\partial U)=0$ for all $U\in\mathcal{B}$ and for every $G$-invariant probability measure $\mu$ on $X$; see \autoref{df:SBP}.}.
	\ee
\end{cnj}

Giol and Kerr \cite{GioKer_subshifts_2010} produced the first examples of minimal, free actions of the integers whose crossed products are not $\mathcal{Z}$-stable, thus showing that $\mathcal{Z}$-stability can indeed fail when $G$ is amenable.
On the other hand, and perhaps surprisingly, we expect $\mathcal{Z}$-stability to be 
automatic for actions of nonamenable groups.

\begin{cnj}[Nonamenable groups]
Let $G\curvearrowright X$ be an amenable, topologically free, minimal action of a nonamenable group. Then $C(X)\rtimes G$ is $\mathcal{Z}$-stable.
\end{cnj}

This survey collects the evidence on both sides, discussing in as much detail as possible the current machinery and techniques, also isolating some of the remaining obstacles and indicating promising directions for future work.

\medskip

\medskip

\textbf{Acknowledgements:} I would like to thank the organizers of the 2025 CIMPA School ``K-theory and Operator Algebras'', with its two week-long sub-schools (I Escuela CMaLP on July 28--August 1, 2025, in La Plata, and XIII Escuela Santal\'o on August 4--8, 2025, in Buenos Aires) for organizing this event and inviting me to deliver one of the minicourses.
Additionally, I would like to thank Paolo Boldrini and Spyros Petrakos for proofreading a preliminary version of these notes and providing very helpful feedback, as well as the anonymous referee for their thorough reading of the manuscript.

Finally, I would like to thank my late master's advisor Fernando Abadie, who unexpectedly passed away a few months ago at a very young age, for being not only my supervisor but also a mentor, a source of inspiration, and my friend. He is the reason why I decided to study C*-algebras, and without his influence, these notes, and some of my work in this area, would not have come to be.

\chapter{Group actions and crossed products}\label{ch:CP}

\section{C*-dynamical systems}\label{sec:DynamSysts}
In this section, we introduce the notion of group action on a \ca, and present
a number of examples of them. Our main source of examples for these notes is topological dynamics,
which will be the focus of the next section.

For a \ca\ $A$, we write $\Aut(A)$ for its automorphism group. If $X$ is any set and $\alpha \colon X\to \Aut(A)$ is a function, in order to avoid cumbersome notation,
we usually write $\alpha_x$ in place of $\alpha(x)$.

\begin{df}
Let $G$ be a discrete group, and let $A$ be a \ca. An \emph{action}
of $G$ on $A$ is a
group homomorphism $\alpha\colon G\to \Aut(A)$.
In this case, we say that the triple $(G,A,\alpha)$ is a \emph{C*-dynamical system}, and
that the pair $(A,\alpha)$ is a $G$-C*-algebra.
\end{df}

The study of group actions is a generalization of the study of automorphisms of \ca s, in view of the
following easy observation.

\begin{rem}
Let $A$ be a \ca. Then there is a one-to-one correspondence between $\Aut(A)$ and $\Z$-actions on $A$.
Indeed, given $\varphi\in\Aut(A)$, the associated action $\alpha^\varphi\colon \Z\to\Aut(A)$ is given
by $\alpha^\varphi_n(a)=\varphi^n(a)$ for all $n\in\Z$ and all $a\in A$, with the convention that $\varphi^0=\id_A$.
\end{rem}

Similarly, there are a one-to-one correspondences:
\bi\item between pairs of commuting automorphisms of $A$ and $\Z^2$-actions on $A$;
\item between pairs of automorphisms of $A$ and $\mathbb{F}_2$-actions on $A$;
\item between automorphisms of $A$ of order $n$ and $\Z_n$-actions on $A$.
\ei

\subsection*{Topological actions}\label{sec:TopActions}
It is well-known that the categories of commutative \ca s and that of locally compact Hausdorff spaces are
equivalent, via the Gelfand transform. As a consequence of functoriality of the Gelfand transform, 
for a fixed discrete group $G$, the
categories of commutative $G$-C*-algebras and that of $G$-topological spaces are also equivalent. We recall
first the definition of a topological action. For a locally compact Hausdorff space $X$, we denote by
$\mathrm{Homeo}(X)$ the group of homeomorphisms of $X$.

\begin{df}
Let $G$ be a discrete group and let $X$ be a locally compact Hausdorff space. An \emph{action} of
$G$ on $X$ is a group homomorphism $\sigma\colon G\to \mathrm{Homeo}(X)$.
When no confusion can arise, we often omit the symbol $\sigma$, and just write $G\curvearrowright X$
to mean that $G$ acts on $X$.
\end{df}

Using the (contravariant) functor given by the Gelfand transform, the following can be shown easily.

\begin{ex}
Let $G$ be a discrete group and let $X$ be a locally compact Hausdorff space.
If $\sigma\colon G \to\mathrm{Homeo}(X)$ is a topological action, then the
formula $\sigma^*_g(f)=f\circ \sigma_g^{-1}$, for $g\in G$ and $f\in C_0(X)$, defines
an action of $G$ on $C_0(X)$.

Moreover, the assignment $\sigma\mapsto \sigma^*$ defines a one-to-one correspondence
between $G$-actions on $X$ and $G$-actions on $C_0(X)$.
\end{ex}

We now give some relevant examples of topological actions.

\begin{egs}\label{egs:TopActions}
Let $G$ be a discrete group.
\be\item There is unique action of $G$ on the one-point space $\{\ast\}$.
More generally, every locally compact Hausdorff space $X$ carries an action of $G$, namely the
trivial on $\id_X\colon G\to\mathrm{Homeo}(X)$.
\item The action $\texttt{Lt}\colon G\to \mathrm{Homeo}(G)$ given by
$\texttt{Lt}_g(h)=gh$ for all $g,h\in G$, is called the \emph{left translation action}.
With a slight abuse of notation, we also write $\texttt{Lt}\colon G\to\Aut(C_0(G))$ for
the induced action.
\item Let $\theta\in \mathbb{R}$ be a number. Then the homeomorphism $r_{\theta}\in \mathrm{Homeo}(S^1)$
given by $r_{\theta}(z)=e^{2\pi i \theta}z$, for all $z\in S^1$, defines a $\Z$-action
on $S^1$. We call this action the \emph{rotation action (by angle $\theta$)}. When $\theta$ is
irrational, we call it an \emph{irrational rotation}.
\item Define the \emph{boundary} $\partial \mathbb{F}_2$ of $\mathbb{F}_2=\langle a,b\rangle$
as the set of right-infinite reduced words on $\{a,a^{-1},b, b^{-1}\}$, and endow it with the
topology in which two infinite words are close if they agree on some long initial segment.
Then $\partial \mathbb{F}_2$ is a Cantor set, and $\mathbb{F}_2$ acts on it by left concatenation
(followed, potentially, by reduction). This is called the \emph{boundary action} of $\mathbb{F}_2$.
\ee
\end{egs}

\section{Covariant representations and crossed products}\label{sec:CovReps}
One of the main goals of these notes is to study the structure of crossed products
by certain classes of actions. This section is devoted to the construction of full and
reduced crossed products.

\begin{df}\label{df:CovRep}
Let $G$ be a discrete group, and let $\alpha\colon G\to\Aut(A)$ be an action on a
\ca\ $A$. A \emph{covariant representation} of $(G,A,\alpha)$ is a triple
$(\Hi,\pi,\varphi)$, where $\Hi$ is a Hilbert space, $\pi\colon G\to \U(\Hi)$ is a
unitary
representation,
and $\varphi\colon A\to \B(\Hi)$ is a representation satisfying
\[\varphi(\alpha_g(a))=\pi_g\varphi(a)\pi_g^*\]
for all $g\in G$ and all $a\in A$.

A covariant representation $(\Hi,\pi,\varphi)$ is said to be \emph{non-degenerate} if
$\varphi$ is non-degenerate.
\end{df}

It is not entirely obvious from the above definition that there
exist \emph{any} covariant representations! There is, however, 
a distinguished
class of covariant representatons, called \emph{regular covariant representations}, which can be explicitly constructed out of any 
representation of the C*-algebra, thus showing that (nontrivial) covariant representations do indeed exist. Regular covariant
representations are also used to define the reduced crossed product,
and are therefore extremely important in our analysis.

For a discrete group $G$ and $g\in G$, we denote by $\mathbbm{1}_g\in\ell^2(G)$ the indicator function of $\{g\}$.
We let $\lambda\colon G\to \mathcal{U}(\ell^2(G))$ denote the 
unitary representation given by $\lambda_g(\mathbbm{1}_h)=\mathbbm{1}_{gh}$ for all $g,h\in G$, and call it the 
\emph{left regular representation} of $G$.

\begin{df}\label{df:RegCovRep}
Let $\varphi\colon A\to \B(\Hi)$ be any representation. The 
\emph{regular covariant representation} associated to $\varphi_0$
is the triple $(\Hi^G,\lambda^\Hi,\varphi^G)$, where
$\Hi^G=\ell^2(G,\Hi)$, and
$\lambda^\Hi\colon G\to\U(\Hi^G)$ and $\varphi^G\colon A\to\B(\Hi^G)$ be given by
\[\lambda^\Hi_g(\xi\mathbbm{1}_h)=\xi \mathbbm{1}_{gh} \ \ \mbox{ and } \ \ \varphi^G(a)(\xi\mathbbm{1}_g)=\varphi(\alpha_{g^{-1}}(a))(\xi)\mathbbm{1}_g\]
for all $g,h\in G$, for all $\xi\in\Hi$ and for all $a\in A$.
\end{df}

\begin{ex} 
Show that $(\Hi^G,\lambda^\Hi,\varphi^G)$ is indeed a covariant
representation, and that it is non-degenerate if and only if $\varphi$ is.
\end{ex}

The full crossed product of a C*-dynamical system $(G,A,\alpha)$ is defined
to be the universal object with respect to all covariant representations of the
system, while the reduced crossed product is defined to be the universal object
with respect to the regular covariant representations; see \autoref{df:FullRedCP} below.
The rigorous definition in the case that $G$ is locally compact requires
that one develops some theory of integration on Banach spaces. Since the focus
of these notes will be on discrete groups, we omit much of this discussion here. For $g\in G$, we regard the indicator function
$\mathbbm{1}_g$ as an element in $c_c(G)\subseteq c_0(G)$. Note that any element of $c_c(G,A)$ is a finite linear
combination of elements of the form $a\mathbbm{1}_g$, for $a\in A$ and $g\in G$.

\begin{df}
Let $G$ be a discrete group, and let $\alpha\colon G\to\Aut(A)$ be an action on a
\ca\ $A$. Endow the vector space $c_c(G,A)$ of finitely supported functions $G\to A$
with the product given by twisted convolution, which is determined by:
\[(a\mathbbm{1}_g)\ast (b\mathbbm{1}_h)=a\alpha_g(b)\mathbbm{1}_{gh}\]
for all $a,b\in A$ and all $g,h\in G$, and with the following twisted convolution
\[(a\mathbbm{1}_g)^*=\alpha_g^{-1}(a^*)\mathbbm{1}_{g^{-1}}\]
for all $a\in A$ and all $g\in G$. Finally, we define the norm of $a\in c_c(G,A)$
by $\|a\|_1=\sum_{g\in G}\|a(g)\|$.
The resulting normed $\ast$-algebra is denoted by $c_c(G,A,\alpha)$.
\end{df}

When $A$ is unital,
we have $\mathbbm{1}_g\ast \mathbbm{1}_h=\mathbbm{1}_{gh}$, that $\mathbbm{1}_g^*=\mathbbm{1}_{g^{-1}}$, and that $\mathbbm{1}_1$ is the unit of 
$c_c(G,A,\alpha)$. In other words, one can regard the assignment $g\mapsto \mathbbm{1}_g$ as a unitary representation of $G$ (in some C*-algebra containing $c_c(G,A,\alpha)$ unitally).

\begin{ex}
Prove that $c_c(G,A,\alpha)$ is a normed $\ast$-algebra, and that it is unital
if and only if $A$ is unital.
\end{ex}

It turns out that $\ast$-representations of $c_c(G,A,\alpha)$ on Hilbert spaces
are in one-to-one correspondence with covariant representations of $(G,A,\alpha)$, as we show next.

\begin{thm}\label{thm:IntForm}
Let $G$ be a discrete group, and let $\alpha\colon G\to\Aut(A)$ be an action on a
\ca\ $A$. Then there is a canonical one-to-one correspondence
between (non-degenerate) covariant representations of $(G,A,\alpha)$ and (non-degenerate)
$\ast$-representations of $c_c(G,A,\alpha)$. If $(\Hi,\pi,\varphi)$ is a covariant representations of $(G,A,\alpha)$,
then the associated $\ast$-representation $\varphi\rtimes\pi\colon  c_c(G,A,\alpha)\to \B(\Hi)$ is called the
\emph{integrated form of $(\Hi,\pi,\varphi)$}, and is given by
\[(\varphi\rtimes\pi)(a\mathbbm{1}_g)=\varphi(a)\pi_g\]
for all $a\in A$ and all $g\in G$.
\end{thm}
\begin{proof}
Let $(\Hi,\pi,\varphi)$ be a covariant representation of $(G,A,\alpha)$. 
We need to show that the formula in the statement determines a $\ast$-representation of $c_c(G,A,\alpha)$.
Given $a,b\in A$ and $g,h\in G$, we have
\begin{align*}
(\varphi\rtimes\pi)\big((a\mathbbm{1}_g)\ast (b\mathbbm{1}_h)\big) &=
(\varphi\rtimes\pi)(a\alpha_g(b)\mathbbm{1}_{gh})\\
&=\varphi(a\alpha_g(b))\pi_{gh}\\
&=\varphi(a)\varphi(\alpha_g(b))\pi_g \pi_h\\
&=\varphi(a)\pi_g \varphi(b)\pi_h\\
&= (\varphi\rtimes\pi)(a\mathbbm{1}_g)\cdot (\varphi\rtimes\pi)(b\mathbbm{1}_h),
\end{align*}
so $\varphi\rtimes\pi$ is multiplicative. One checks analogously that $\varphi\rtimes\pi$ is 
$\ast$-preserving. 

We proceed to show that any $\ast$-representation of $c_c(G,A,\alpha)$ is the integrated form
of a covariant representation of $(G,A,\alpha)$.
For the sake of the argument, we will assume that $A$ is unital from now on.

We denote by $\iota\colon A\to c_c(G,A,\alpha)$ 
the $\ast$-homomorphism
given by $\iota(a)=a\mathbbm{1}_1$ for $a\in A$.
Let $\psi\colon c_c(G,A,\alpha)\to \B(\Hi)$ be a $\ast$-representation. 
Set $\varphi=\psi\circ\iota$ and let $\pi\colon G\to \B(\Hi)$ be given by $\pi_g=\psi(\mathbbm{1}_g)$ for
all $g\in G$. Then $\varphi$ is a representation of $A$, and $\pi$ is a unitary representation of $G$.

We claim that $\psi=\varphi\rtimes\pi$. To check this, it suffices to show that they agree
on elements of the form $a\mathbbm{1}_g$, for $a\in A$ and $g\in G$. This is readily checked:
\[(\varphi\rtimes\pi)(a\mathbbm{1}_g)=\varphi(a)\pi_g=\psi(a\mathbbm{1}_1)\psi(\mathbbm{1}_g)=\psi(a\mathbbm{1}_g),\]
thus concluding the proof.
\end{proof}

\begin{ex} Complete the proof of the above result by proving it in the case that $A$ is not necessarily unital. \emph{(Hint: use limits along an approximate identity to define
the unitary representation.)}
\end{ex}

We are now ready to introduce full and reduced crossed products.

\begin{df}\label{df:FullRedCP}
Let $G$ be a discrete group, and let $\alpha\colon G\to\Aut(A)$ be an action on a
\ca\ $A$.

We define the \emph{full crossed product} $A\rtimes_{\alpha} G$ of $(G,A,\alpha)$ to be the completion of
$c_c(G,A,\alpha)$ with respect to the norm
\[\|a\|_{A\rtimes_{\alpha} G} = \sup \Big\{\big\|(\varphi\rtimes\pi)(a)\big\|\colon (\varphi, \pi) \mbox{ is a covariant representation of } (G,A,\alpha)\Big\}.\]

We define the \emph{reduced crossed product} $A\rtimes_{\lambda,\alpha} G$ of $(G,A,\alpha)$ to be the completion of
$c_c(G,A,\alpha)$ with respect to the norm
\[\|a\|_{A\rtimes_{\lambda,\alpha} G} = \sup \Big\{\big\|(\varphi^G\rtimes \lambda^\Hi)(a)\big\|\colon \varphi\colon A\to \mathcal{B}(\Hi) \mbox{ is a representation of } A\Big\}.\]
\end{df}


It follows directly from the definitions that representations of $A\rtimes_{\alpha} G$ are
in one-to-one correspondence with covariant representations of $(G,A,\alpha)$.
This is a generalization of the corresponding fact for the full group \ca\ 
$C^*(G)$, which is the full crossed
product of the trivial action of $G$ on $\mathbb{C}$. 
Similarly, the reduced group C*-algebra $C^*_{\lambda}(G)$ is the reduced crossed product of the trivial action of $G$ on $\mathbb{C}$.

\autoref{sec:EgsCP} contains the explicit
computations of a number of crossed products.
In particular, \autoref{eg:TrivialAction}
is a generalization of the observation that $\C\rtimes G= C^*(G)$ and $\C\rtimes_{\lambda} G=C^*_{\lambda}(G)$.

As it turns out, 
the full crossed product of a C*-dynamical system can be 
completely described
through a universal property intimately related to covariant representations of the system. This presentation is
not surprising, but it has the advantage of making certain computations of crossed products much easier.

\begin{thm}\label{thm:UnivPropCP}
Let $G$ be a discrete group, let $A$ be a \ca, and let $\alpha\colon G\to\Aut(A)$ be
an action. Then there exist homomorphisms
\[\iota_A\colon A\to M(A\rtimes_{\alpha} G) \ \mbox{ and } \ \iota_G\colon G\to \U(M(A\rtimes_{\alpha} G))\]
satisfying
\be
\item[(i)] $\iota_A(\alpha_g(a))=\iota_G(g)\iota_A(a)\iota_G(g)^{-1}$ for all $a\in A$ and all $g\in G$;
\item[(ii)] $A\rtimes_{\alpha} G=\overline{\mathrm{span}}\{\iota_A(a)\iota_G(g)\colon a\in A, g\in G\}$;
\item[(iii)] Whenever $(\Hi, \pi,\varphi)$ is a covariant representation of $(G,A,\alpha)$, then
\[\varphi=(\varphi\rtimes\pi)\circ\iota_A \ \mbox{ and } \pi=(\varphi\rtimes\pi)\circ\iota_G.\]
\ee

Moreover, suppose that $B$ is another \ca\ and that $j_A\colon A\to M(B)$ and $j_G\colon G\to \U(M(B))$ are maps
satisfying conditions (i), (ii) and (iii) above. Then there exists an isomorphism $\psi\colon A\rtimes_{\alpha} G\to B$
such that $\psi\circ\iota_A=j_A$ and $\psi\circ \iota_G=j_G$.
\end{thm}
\begin{proof}
For $g\in G$, let $\iota_G(g)\colon c_c(G,A,\alpha)\to c_c(G,A,\alpha)$ be given by 
\[\iota_G(g)(\xi)(h)=\alpha_g\big(\xi(g^{-1}h)\big)\]
for all $\xi\in C_c(G,A,\alpha)$ and all $h\in G$. It is easily seen that $\iota_G(g)$ determines a multiplier of
$A\rtimes_{\alpha} G$, and that the resulting map $\iota_G\colon G\to M(A\rtimes_{\alpha} G)$ is a unitary representation.

For $a\in A$, let $\iota_A(a)\colon c_c(G,A,\alpha)\to c_c(G,A,\alpha)$ be given by $\iota_A(a)(\xi)=a\xi$ for all $\xi\in c_c(G,A,\alpha)$.
It is easily seen that $\iota_A(a)$ determines an invertible multiplier of
$A\rtimes_{\alpha} G$, and that the resulting map $\iota_A\colon A\to M(A\rtimes_{\alpha} G)$ is a homomorphism. The maps
$\iota_G$ and $\iota_A$ clearly satisfy conditions (i), (ii) and (iii) above.

Now let $(B,j_A,j_G)$ be a triple as in the statement. Since $(j_A,j_G)$ is a covariant representation and $j_A$ is non-degenerate,
there exists a non-degenerate homomorphism $j_A\rtimes j_G\colon A\rtimes_{\alpha} G\to B$ satisfying
\[(j_A\rtimes j_G)\circ \iota_A=j_A \ \mbox{ and } \ (j_A\rtimes j_G)\circ\iota_G=j_G.\]
Reversing the roles of $(B,j_A,j_G)$ and $(A\rtimes_{\alpha} G, \iota_A,\iota_G)$ gives a non-degenerate homomorphism
$\iota_A\rtimes \iota_G\colon B\to A\rtimes_{\alpha} G$ satisfying
\[(\iota_A\rtimes \iota_G)\circ j_A=\iota_A \ \mbox{ and } \ (\iota_A\rtimes \iota_G)\circ j_G=\iota_G.\]
It is then immediate that $j_A\rtimes j_G$ and $\iota_A\rtimes \iota_G$ are mutual inverses.
\end{proof}

\begin{rem}
Instead of developing the theory of crossed products as we did in \autoref{sec:CovReps}, we could have defined crossed products
through the universal property from the previous theorem. Had we taken that route, we would have had to show that there
exists a least one (and hence precisely one) crossed product. Of course our work in \autoref{sec:CovReps} shows that one crossed product
can be obtained by completing $c_c(G,A,\alpha)$ with respect to its universal norm, but there is a slightly shorter way
of proving its existence. We can find a set $\mathcal{S}$ of covariant representations of $(G,A,\alpha)$ such that
every cyclic covariant representation is equivalent to an element of $\mathcal{S}$. Define representations
\[\iota_A=\bigoplus_{(\Hi,\pi,\varphi)\in\mathcal{S}} \varphi \ \ \mbox{ and } \ \ \iota_G=\bigoplus_{(\Hi,\pi,\varphi)\in\mathcal{S}} \pi,\]
of $A$ and $G$, respectively, on the Hilbert space $\bigoplus\limits_{(\Hi,\pi,\varphi)\in\mathcal{S}} \Hi$.
One can then take $A\rtimes_{\alpha} G$ to be the closed linear span of $\{\iota_A(a)\iota_G(\mathbbm{1}_g)\colon a\in A, g\in G\}$.
\end{rem}

As a consequence of their universal property, 
full crossed products by discrete groups admit a very natural description as universal \ca s
with generators and relations, as we show next.

\begin{thm}\label{thm:FullCPDiscreteUniv}
Let $G$ be a discrete group, and let $\alpha\colon G\to\Aut(A)$ be an action on a
\ca\ $A$. Then $A\rtimes_{\alpha} G$ is canonically isomorphic to the universal
\ca\ generated by the set $\{au_g\colon a\in A, g\in G\}$, where $u_g$ is a unitary for
all $g\in G$, subject to the relations $au_gbu_h=a\alpha_g(b)u_{gh}$ for all $a,b\in A$
and all $g,h\in G$.
\end{thm}
\begin{proof}
Denote by $B$ the universal \ca\ described in the statement; we will show that
there is a canonical isomorphism $B\cong A\rtimes_{\alpha} G$. 
By universality of $B$, and since $A\rtimes_{\alpha} G$
is generated by the same generators and relations as $B$, there is a surjective map
$\sigma\colon B\to A\rtimes_{\alpha} G$, which maps an element $au_g$ to 
$a\mathbbm{1}_g\in c_c(G,A,\alpha)$, for all $a\in A$ and all $g\in G$. 

Observe that there is a canonical inclusion $\iota\colon c_c(G,A,\alpha)\to B$ determined 
by sending $\iota(a\mathbbm{1}_g)=au_g$ for all $a\in A$ and all $g\in G$. 
We obtain a $\ast$-representation of $c_c(G,A,\alpha)$ on $B$,
which by \autoref{thm:IntForm} must be the integrated form of some covariant representation
$(\pi,\varphi)$, where $\varphi\colon A\to B$ is a homomorphism and $\pi\colon G\to M(B)$
is a unitary representation. By the universal property of the crossed product (see \autoref{thm:UnivPropCP}), the
integrated form of $(\pi,\varphi)$ extends to a homomorphism $\theta\colon A\rtimes_{\alpha} G\to B$
satisfying 
\[\theta(a\mathbbm{1}_g)=(\varphi\rtimes\pi)(a\mathbbm{1}_g)=au_g.\] 
It follows that $\sigma$ and $\theta$ are mutual inverses, and thus $B$ is canonically isomorphic to
$A\rtimes_{\alpha} G$.
\end{proof}

\begin{rem}\label{rem:Zcrossedprod}
In particular, when $A$ is unital, a crossed product of the form $A\rtimes_{\alpha} \Z$
is the universal \uca\ generated by a unital copy of $A$ together with a unitary $u$
satisfying $uau^*=\alpha_1(a)$.
\end{rem}

Recall that a discrete group $G$ is said to be \emph{amenable} if for every finite subset $K \subseteq G$ and every $\varepsilon > 0$, there exists a finite subset $F \subseteq G$ (called a $(K,\varepsilon)$-invariant set) such that
\[
\frac{|FK \triangle F|}{|F|} < \varepsilon.
\]
Here, $\triangle$ denotes the symmetric difference. Although we will not need this here, we mention that a sequence $(F_n)_{n \in \mathbb{N}}$ of finite subsets of $G$ is said to be a \emph{F\o lner sequence} if for every finite subset $K \subseteq G$ and every $\varepsilon > 0$, there exists $n \in \mathbb{N}$ such that $F_n$ is $(K,\varepsilon)$-invariant.

\begin{ex}
	Show that a countable group is amenable if and only if it admits a F\o lner sequence.
\end{ex}

By construction, there is always a (canonical) 
surjective homomorphism
$\kappa\colon A\rtimes G\to A\rtimes_{\lambda} G$, 
but it may fail to be injective\footnote{We can at this point explain why we will mostly be interested in \emph{reduced} crossed products: soon we will want all our crossed products to be simple, and since the full crossed product always has the reduced one as a quotient, it is natural to explore simplicity of the latter. For example, a full group C*-algebra is never simple, but reduced group C*-algebras may be simple, such as for $\mathbb{F}_n$.}. 
A very general instance in which full and
reduced crossed products agree is when the acting group is \emph{amenable}, as we show in the
next theorem\footnote{The proof is very similar to the proof that
the universal map $C^*(G)\to C^*_{\lambda}(G)$ is an isomorphism when $G$ is amenable.}.

\begin{thm}\label{thm:AmenGFullRedCP}
Let $G$ be a discrete group, let $A$ be a \ca, and let $\alpha\colon G\to\Aut(A)$
be an action. If $G$ is amenable, then the natural map
$\kappa\colon A\rtimes_{\alpha} G\to A\rtimes_{\lambda,\alpha}G$ is an isomorphism.
\end{thm}
\begin{proof}
Let $x\in c_c(G,A,\alpha)$, let $(\Hi,\pi,\varphi)$ be a covariant representation, and let $\ep>0$.
Use $\varphi$ to construct the regular covariant representation $(\Hi^G,\lambda^\Hi,\varphi^G)$
as in \autoref{df:RegCovRep}. Using the notation from \autoref{thm:IntForm}, we will show that
\[\|(\varphi\rtimes\pi)(x)\| -\ep \leq \|(\varphi^G\rtimes \lambda^\Hi)(x)\|.\]

Recall that $\Hi^G$ is just $\ell^2(G,\Hi)$. Let $v\in \U(\Hi^G)$ be the unitary given by
\[v(\xi)(g)=\pi_g^{-1}(\xi(g))\]
for all $\xi\in \Hi^G$ and all $g\in G$.

It is straightforward to verify that for all $g\in G$ and all $a\in A$, we have
\[v(\lambda_g\otimes \pi_g)v^*=\lambda_g\otimes 1_{\Hi}, \ \ \mbox{ and } \ \ v(1_{\ell^2(G)}\otimes \varphi(a))v^*=\varphi^G(a).\]

It follows that $(\Hi^G,\lambda\otimes\pi,1_{\ell^2(G)}\otimes \varphi)$ is a covariant representation
that induces the same norm as $(\Hi^G, \lambda^\Hi,\varphi^G)$ on $c_c(G,A,\alpha)$. It therefore suffices
to show that
\[\|(\varphi\rtimes\pi)(x)\| -\ep \leq \|((1_{\ell^2(G)}\otimes\varphi)\rtimes (\lambda\otimes\pi))(x)\|.\]

Without loss of generality, we may assume that $\|(\varphi\rtimes\pi)(x)\|=1$.
Choose $\xi_0\in \Hi$ with $\|\xi_0\|=1$ and
\[\|(\varphi\rtimes\pi)(x)\xi_0\|>1-\frac{\ep}{2}.\]
Set $K=\supp(x)\cup\{1\}$ and
\[\delta=\left(\frac{1-\frac{\ep}{2}}{1 -\ep}\right)^2-1>0,\]
and use amenability of $G$ to find a finite subset $F\subseteq G$ satisfying $|K^{-1}F|<(1+\delta)|F|$.
Set $\xi=\xi_0\mathbbm{1}_{K^{-1}F}\in \Hi^G=\ell^2(G,\Hi)$.
Then 
\[\|\xi\|=|K^{-1}F|^{1/2}\|\xi_0\|<(1+\delta)^{1/2}|F|^{1/2}.\] 
One shows that
\[\Big((1\otimes\varphi)\rtimes (\lambda\otimes\pi)\Big)(x)\xi(g)= (\varphi\rtimes\pi)(x)(\xi_0)\]
for all $g\in F$, from which it follows that
\[\Big\|\Big((1\otimes\varphi)\rtimes (\lambda\otimes\pi)\Big)(x)\xi\Big\|\geq |F|^{1/2} \|(\varphi\rtimes\pi)(x)(\xi_0)\|\geq |F|^{1/2}
\left(1-\frac{\ep}{2}\right).\]
We conclude that
\begin{align*}
\Big\|\Big((1\otimes\varphi)\rtimes (\lambda\otimes\pi)\Big)(x)\Big\| &> \frac{ |F|^{1/2} \left(1-\frac{\ep}{2}\right)}{(1+\delta)^{1/2}|F|^{1/2}}\\
&= (1+\delta)^{-1/2} \left(1-\frac{\ep}{2}\right)=\left(1-\ep\right),
\end{align*}
as desired.
\end{proof}

When $A=\C$, the converse to \autoref{thm:AmenGFullRedCP} is true: if $C^*(G)=C^*_{\lambda}(G)$, then
$G$ is amenable. For actions on arbitrary \ca s, however, this need not be the case. For example,
$C_0(G)\rtimes G =C_0(G)\rtimes_{\lambda} G$ regardless of $G$, as we will see in \autoref{prop:CPLeftTransl}.

\section{Examples and computations}

\label{sec:EgsCP}

This section is devoted to the computation of a number of examples of full and
reduced crossed products. We begin with the trivial action on a \ca.

\begin{eg}\label{eg:TrivialAction}
Let $A$ be a \ca, let $G$ be a discrete group, and let $\id_A\colon G\to\Aut(A)$ be the
trivial action. Then there are canonical isomorphisms
\[A\rtimes_{\id_A}G\cong A\otimes_{\mathrm{max}} C^*(G) \ \mbox{ and } \ A\rtimes_{\lambda, \id_A}G\cong A\otimes_{\mathrm{min}} C_{\lambda}^*(G).\]
\end{eg}
\begin{proof}
Covariant representations for $(G,A,\id_A)$ on a Hilbert space $\Hi$ are in natural one-to-one
correspondence with pairs $(\varphi,\pi)$ of commuting representations $\varphi\colon A\to \B(\Hi)$
and $\pi\colon C^*(G)\to \B(\Hi)$. Thus, the universal object for these representations is the
maximal tensor product of $A$ and $C^*(G)$. Similarly, regular covariant representations of
$(G,A,\id_A)$ have the form $(\varphi\otimes\id_{\ell^2(G)},\id_\Hi\otimes \lambda)$,
where $\varphi\colon A\to\B(\Hi)$ is some representation and $\lambda\colon G\to \U(\ell^2(G))$ is
the left regular representation. The universal \ca\ that these representations generate is
thus the minimal tensor product of $A$ and $C^*_{\lambda}(G)$.
\end{proof}

The notation $\rtimes$ for crossed products is inspired by the following fact.

\begin{ex}
Let $G=N\rtimes H$ be a semidirect product of discrete groups. Show that there are canonical
isomorphisms
\[C^*(G)\cong C^*(N)\rtimes H \ \ \mbox{ and } \ \ C_{\lambda}^*(G)\cong C_{\lambda}^*(N)\rtimes_{\lambda} H.\]
\end{ex}

For the next computation, we recall that $M_n$ is the universal \uca\ generated by
$\{e_{j,k}\colon j,k=1,\ldots,n\}$,
subject to the relations
\[e_{j,k}e_{l,m}=\delta_{k,l}e_{j,m}, \ \ e_{j,k}^*=e_{k,j}, \ \ \sum_{j=1}^n e_{j,j}=1\]
for all $j,k,l,m=1,\ldots,n$. Elements $e_{j,k}$ as above are called \emph{matrix units} for $M_n$.

\begin{eg}\label{prop:CPLeftTransl}
Let $G$ be a discrete group. Then there are canonical isomorphisms
\[c_0(G)\rtimes_{\texttt{Lt}}G\cong c_0(G)\rtimes_{\lambda,\texttt{Lt}}G\cong \K(\ell^2(G)).\]
\end{eg}
\begin{proof}
For $g\in G$, we denote by $u_g$ the canonical unitary implementing $\texttt{Lt}_g$ in the full
crossed product, and we denote by $\mathbbm{1}_g\in c_0(G)$ the characteristic function of $\{g\}$.
For $g,h\in G$, set $e_{g,h}=\mathbbm{1}_gu_{gh^{-1}}\in c_0(G)\rtimes_{\texttt{Lt}}G$.
Let $F\subseteq G$ be a finite set.

\textbf{Claim:} \emph{There is a canonical isomorphism 
\[\varphi_F\colon \B(\ell^2(F))\to C^*(\{e_{g,h}\colon g,h\in F\})\]
satisfying $\varphi_F(1)=\sum_{g\in F}e_{g,g}$.}
For this, it is enough to show that the set $\{e_{g,h}\colon g,h\in F\}$ satisfies the relations for matrix
units in $B(\ell^2(F))$. For $g,h,k,m\in F$, we have
\begin{align*}
e_{g,h}e_{k,m}&=\mathbbm{1}_gu_{gh^{-1}}\mathbbm{1}_ku_{km^{-1}}\\
&= \mathbbm{1}_g \texttt{Lt}_{gh^{-1}}(\mathbbm{1}_{k})u_{gh^{-1}km^{-1}}\\
&= \mathbbm{1}_g\mathbbm{1}_{gh^{-1}k}u_{gh^{-1}km^{-1}}\\
&= \mathbbm{1}_g \delta_{h,k} u_{gm^{-1}}.
\end{align*}
On the other hand, it is clear that $e_{g,h}^*=e_{h,g}$ for all $g,h\in F$.
The claim thus follows.

Observe that if $F_1\subseteq F_2$ are finite subsets of $G$, and we regard $\B(\ell^2(F_1))$ canonically as a
(non-unital) subalgebra of $\B(\ell^2(F_2))$, then $\varphi_{F_2}|_{\B(\ell^2(F_1))}=\varphi_{F_1}$.
We deduce that the family $(\varphi_F)_{F\subseteq G \text{ finite}}$ induces a well-defined
contractive homomorphism
\[\varphi\colon \bigcup_{F\subseteq G \text{ finite}} \B(\ell^2(F))\to c_c(G,c_0(G),\texttt{Lt}).\]
By continuity, $\varphi$ extends to a homomorphism $\psi\colon \K(\ell^2(G))\to c_0(G)\rtimes_{\texttt{Lt}}G$.
This homomorphism is surjective, since its range contains the dense set $\mathrm{span}\{e_{g,h}\colon g,h\in G\}$.
Since $\K(\ell^2(G))$ is simple, $\psi$ must also be injective and hence an isomorphism.

Finally, since $c_0(G)\rtimes_{\texttt{Lt}}G$ is simple and $c_0(G)\rtimes_{\lambda,\texttt{Lt}}G$
is a nonzero quotient of it (via $\kappa$), they must be isomorphic (via $\kappa$).
\end{proof}

The previous proposition can be generalized as follows.

\begin{ex}\label{prop:GactsGH}
Let $G$ be a discrete group, let $H$ be a subgroup, and let $G$ act on $G/H$ by translation.
Prove that there are canonical isomorphisms
\[c_0(G/H)\rtimes_{\texttt{Lt}}G\cong C^*(H)\otimes \K(\ell^2(G/H))\]
and
\[c_0(G/H)\rtimes_{\lambda,\texttt{Lt}}G\cong C_{\lambda}^*(H)\otimes \K(\ell^2(G/H)).\]\end{ex}

\begin{eg}\label{eg:IrrRotAlg}
Let $\theta\in \mathbb{R}$ be a number, and let $r_{\theta}\colon \Z\to\Aut(C(S^1))$
denote the rotation action from part~(3) of \autoref{egs:TopActions}. Then $C(S^1)\rtimes_{r_{\theta}}\Z$
is isomorphic to $A_{\theta}$, that is, the universal \ca\ generated by two unitaries
$u,v$ satisfying $uv=e^{2\pi i\theta}vu$.
\end{eg}
\begin{proof}
Note that $\Z$ is amenable, and thus full and reduced crossed products agree by \autoref{thm:AmenGFullRedCP}.
Let $v$ denote the canonical generating unitary of $C(S^1)$. Then $r_{\theta}(v)=e^{\pi i \theta}v$,
and hence the result follows from \autoref{rem:Zcrossedprod}. \end{proof}

\begin{eg}\label{Eg:Z2S1ConjugationCP}
Let $\Z_2$ act on $S^1$ via conjugation, and denote by $\alpha\colon \Z_2\to\Aut(C(S^1))$ the
induced action. We will show that
\[C(S^1)\rtimes_{\alpha} \Z_2\cong \Big\{f\in C([-1,1],M_2)\colon f(1),f(-1) \mbox{ are diagonal}\Big\}.\]

Denote by $S^1_+$ the closed upper-half circle, and observe first that the algebra
\[B=\left\{f\in C(S^1_+,M_2)\colon f(1),f(-1) \mbox{ are of the form }
\begin{bmatrix}
    x & y \\
    y & x
\end{bmatrix}\right\}\]
is isomorphic to what we want to get. Define $\varphi\colon C(S^1)\to B$ by
\[\varphi(f)=\begin{bmatrix}
    f(z) & 0 \\
    0 & f(\overline{z})
\end{bmatrix}\]
for all $f\in C(S^1)$ and all $z\in S^1_+$, and let $v\in \U(B)$ be the unitary given
by $v(z)=\begin{bmatrix}
    0 & 1 \\
    1 & 0
\end{bmatrix}$ for all $z\in S^1_+$. One checks that $v^2=1$ and that
$v\varphi(f)v^*=\varphi(\alpha_1(f))$ for all $f\in C(S^1)$. By the universal
property of $C(S^1)\rtimes_{\alpha} \Z_2$ (see \autoref{thm:UnivPropCP}), there is a unique homomorphism
$\psi\colon C(S^1)\rtimes_{\alpha} \Z_2 \to B$ extending $\varphi$ and
satisfying $\psi(u)=v$. One checks that this homomorphism is given by
\[\psi(f_0+f_1u)=\begin{bmatrix}
    f_0(z) & f_1(z) \\
    f_1(\overline{z}) & f_0(\overline{z})
\end{bmatrix}\]
for all $f_0,f_1\in C(S^1)$ and all $z\in S^1_+$.
The rest of the proof consists in showing that $\psi$ is an isomorphism, which
we leave as an exercise.
\end{eg}

\begin{ex}\label{Ex:Z2S1ConjugationCP}
Complete the details in \autoref{Eg:Z2S1ConjugationCP}.
\end{ex}

\begin{ex}\label{Ex:Z2IntervalOppositeCP}
Let $\Z_2$ act on $[-1,1]$ via multiplication by -1, and denote by $\alpha\colon \Z_2\to\Aut(C([-1,1]))$ the
induced action. Compute $C([-1,1])\rtimes_{\alpha} \Z_2$.
\emph{(Hint: realize $\Z_2\curvearrowright [-1,1]$ as an equivariant
quotient of the conjugation action $\Z_2\curvearrowright S^1$ from 
\autoref{Eg:Z2S1ConjugationCP}.)}
\end{ex}

\section{The canonical conditional expectation}\label{sec:RedCPDiscrete}

In this section, we will establish the existence of a canonical conditional expectation
from full and reduced crossed products by discrete groups back into the original coefficient algebra.
We will also prove a distinguishing feature of the reduced crossed product, namely, faithfulness of its
canonical conditional expectation.

\begin{nota} 
 Let $G$ be a discrete group, let $A$ be a unital \ca, and let $\alpha\colon G\to\Aut(A)$ be an action. Since the image of $\mathbbm{1}_g\in c_c(G,A,\alpha)$ in $A\rtimes_{\alpha} G$ is a unitary, it is most common to denote it by $u_g$. The same considerations 
 apply when $A$ is not unital, in which case $u_g$ is canonically
 a unitary in $M(A\rtimes_{\alpha} G)$.
\end{nota}

\begin{lma}\label{lma:FormulasFaithCondExp}
Let $G$ be a discrete group, let $A$ be a \ca, and let $\alpha\colon G\to\Aut(A)$ be an action.
Let $\varphi\colon A\to \B(\Hi)$ be a representation and let $(\Hi^G,\lambda^\Hi, \varphi^G)$ be its
associated regular covariant representation as in \autoref{df:RegCovRep}. Let $F\subseteq G$ be
a finite set, let $a_g\in A$, for $g\in F$, and set
\[a=\sum_{g\in F}a_gu_g\in c_c(G,A,\alpha)\subseteq A\rtimes_{\lambda,\alpha} G.\]
\be\item For $\xi\in\Hi$ and $g\in G$, we have
\[(\varphi^G\rtimes \lambda^{\Hi})(a)(\xi\mathbbm{1}_g)= \varphi(\alpha_{g^{-1}}(a_g))(\xi)\mathbbm{1}_g.\]
\item For $g\in G$, let $s_g\in \B(\Hi, \Hi^G)$ be the isometry which sends $\xi$ to $\xi\mathbbm{1}_g$, for
all $\xi\in\Hi$. For all $g,h\in G$, we have
\[s_g^*(\varphi^G\rtimes \lambda^{\Hi})(a)s_h=\varphi(\alpha_{g^{-1}}(a_{gh^{-1}})).\]\ee
\end{lma}
%

The proof of the previous lemma is routine, so we leave it as an exercise.

\begin{ex}
Prove \autoref{lma:FormulasFaithCondExp}.
\end{ex}

\begin{lma}\label{lma:DefEg}
Let $G$ be a discrete group, let $A$ be a \ca, let $\alpha\colon G\to\Aut(A)$ be an action,
and fix $g\in G$. Then there exists a unique contractive linear map $E_g\colon A\rtimes_{\lambda,\alpha}G\to A$
satisfies $E_g(a)=a_g$ for all $a=\sum_{h\in F} a_hu_h\in c_c(G,A,\alpha)$ as in \autoref{lma:FormulasFaithCondExp}.
Moreover, for every representation $\varphi\colon A\to\B(\Hi)$, one has
\[s_g^*(\varphi^G\rtimes\lambda^\Hi)(b)s_h=\varphi(\alpha_{g^{-1}}(E_{gh^{-1}}(b)).\]
for all $g,h\in G$, and for all $b\in A\rtimes_{\lambda,\alpha}G$. \end{lma}
\begin{proof}
Define a linear map $E_g^{(0)}\colon c_c(G,A,\alpha)\to A$ by evaluating at $g$. Then $E_g$ is continuous
with respect to the $\|\cdot\|_{\infty}$ norm on $c_c(G,A,\alpha)$. It therefore suffices to show that
$\|\cdot\|_{\infty}\leq \|\cdot\|_{\lambda}$. For $a$ as in the statement, for $h\in G$, and for an
isometric representation $\varphi\colon A\to\B(\Hi)$, we have
\[\|a_h\|=\|\varphi(a_h)\|=\|s_1^*(\varphi^G\rtimes\lambda^\Hi)(a)s_{h^{-1}}\|\leq \|(\varphi^G\rtimes\lambda^\Hi)(a)\|\leq \|a\|_{\lambda}.\]
It follows that $\|\cdot\|_{\infty}\leq \|\cdot\|_{\lambda}$, as desired.
The last identity in the statement follows immediately from part~(2) of \autoref{lma:FormulasFaithCondExp}.
\end{proof}

It follows that for any $a\in A\rtimes_{\lambda,\alpha}G$, and hence for every $a\in A\rtimes_{\alpha} G$, there
are well-defined coefficients $a_g\in A$, for $g\in G$. (This is obvious for elements in $c_c(G,A,\alpha)$.)
However, $a$ cannot in general be recovered as $\sum_{g\in G}a_gu_g$, since this series does not in general
converge in the crossed product. On the other hand, as we show next, an element in the reduced crossed product is zero if and only if
all of its coefficients are zero.

\begin{thm}\label{thm:CondExpFaithful}
Let $G$ be a discrete group, let $A$ be a \ca, and let $\alpha\colon G\to\Aut(A)$ be an action.
\be\item If $a\in A\rtimes_{\lambda,\alpha}G$ satisfies $E_g(a)=0$ for all $g\in G$, then $a=0$.
\item If $a\in A\rtimes_{\lambda,\alpha}G$, then $\|E_g(a)\|^2\leq \|E_1(a^*a)\|$ for all $g\in G$.
\item If $a\in A\rtimes_{\lambda,\alpha}G$ satisfies $E_1(a^*a)=0$, then $a=0$.
\ee
\end{thm}
\begin{proof}
(1). Let $\varphi\colon A\to\B(\Hi)$ be a representation. If
$a\in A\rtimes_{\lambda,\alpha}G$ satisfies $E_g(a)=0$ for all $g\in G$, then
$s_g^*(\varphi^G\rtimes\lambda^{\Hi})(a)s_h=0$ for all $g,h\in G$, and hence
$(\varphi^G\rtimes\lambda^{\Hi})(a)=0$. Since $\varphi$ is arbitrary, it follows that $a=0$.

(2). For a finite subset $F\subseteq G$ and an element $a=\sum_{g\in F}a_gu_g\in c_c(G,A,\alpha)$, we have
\[a^*a=\sum_{g,h}\alpha_{g^{-1}}(a_ga_h^*)u_{g^{-1}h}\]
and thus
\[E_1(a^*a)=\sum_{g\in G}a_ga_g^*=\sum_{g\in G}E_g(\alpha_{g^{-1}}(a))E_g(\alpha_{g^{-1}}(a)^*).\]
Thus $E_1(a^*a)\geq \alpha_{g^{-1}}(E_g(a)E_g(a)^*)$ for all $g\in G$ and hence
$\|E_1(a^*a)\|\geq \|E_g(a)\|^2$. By continuity, this holds for all elements in $A\rtimes_{\lambda,\alpha}G$.

(3). If $E_1(a^*a)=0$, then $E_g(a)=0$ for all $g\in G$ by part~(2), and hence $a=0$ by part~(1).
\end{proof}

Among the maps $E_g$, for $g\in G$, the one corresponding to the unit of the group is
special. This is already evident in part~(3) of the above theorem. Moreover, $E_1$ is a completely positive map, and it is even a conditional expectation
in the sense of the following definition.

\begin{df}\label{df:CondExp}
Let $B$ be a \ca, let $A\subseteq B$ be a C*-subalgebra, and let $E\colon B\to A$ be a positive, linear map.
We say that $E$ is a \emph{conditional expectation} if $E(a)=a$ for all $a\in A$,
and $E(ab)=aE(b)$ for all $a\in A$ and all $b\in B$.
\end{df}

\begin{prop}\label{prop:E1CondExp}
Let $G$ be a discrete group, let $A$ be a \ca, and let $\alpha\colon G\to\Aut(A)$ be an action.
Then the map $E_1\colon A\rtimes_{\lambda,\alpha}G\to A$ is a conditional expectation.
\end{prop}
\begin{proof}
Positivity of $E_1$ follows by taking $g=h=1$ in the displayed identity of \autoref{lma:DefEg}.
Since $A$ is identified with $Au_1\subseteq A\rtimes_{\lambda,\alpha}G$, it is clear that $E_1(a)=a$
for all $a\in A$. Finally, given $a\in A$ and $b\in c_c(G,A,\alpha)$ of the form
$b=\sum\limits_{g\in G}b_gu_g$, it is clear that
$E_1(ab)=ab_1=aE_1(b)$. Since $E_1$ is continuous, it follows that this
identity holds for all $b\in A\rtimes_{\lambda,\alpha} G$, as desired.
\end{proof}

\begin{df}\label{df:CondExpCP}
Let $G$ be a discrete group, let $A$ be a \ca, and let $\alpha\colon G\to\Aut(A)$ be an action.
Then the maps
\[E=E_1\colon A\rtimes_{\lambda,\alpha}G\to A \ \ \mbox{ and } \ \ E\circ\kappa\colon A\rtimes_{\alpha} G\to A\]
are called
the \emph{canonical conditional expectations} on the reduced and full crossed products, respectively.
\end{df}

There are other conditional expectations
in crossed products that sometimes play a relevant
role. We leave the following as an exercise.

\begin{ex}
Let $G$ be a discrete group, let $X$ be a compact Hausdorff space, let $\alpha\colon G\to\Aut(C(X))$ be an action, and let $x\in X$.
Set 
\[G_x=\{g\in G\colon g\cdot x=x\}.\]
Show that there exists a conditional expectation
$E_x\colon C(X)\rtimes_{\lambda} G\to C^*_{\lambda}(G_x)$ determined by
\[ E_x(au_g)=\begin{cases}
a(x)u_g, &\mbox{ if } g\in G_x;\\
0, &\mbox{ else}.
                 \end{cases}\]
\end{ex}

\chapter{Crossed products within Elliott's program}
 
\section{Specializing the Elliott program to dynamical systems} 
We now give a very brief overview of Elliott's program, 
focusing on the classification theorem, and we refer the 
reader to \autoref{app:Elliott} for a historical account
as well as some earlier results in the area.

For the last four decades, one of the main driving forces in the study of C*-algebras has been the Elliott
classification program, which aims at classifying simple, separable, nuclear C*-algebras which satisfy the
Universal Coefficient Theorem (UCT) using K-theoretical and tracial data (known as the Elliott invariant).

Despite the success that this program witnessed in the 90’s, examples constructed by R{\o}rdam \cite{Ror_simple_2003} and Toms
\cite{Tom_classification_2008} (inspired by earlier work of Villadsen; see Section~4 in \autoref{app:Elliott}) showed that an additional regularity assumption for the C*-algebras ought to be added. 
In parallel to these developments,
Jiang and Su constructed in \cite{JiaSu_simple_1999} 
an infinite-dimensional, 
unital, simple,
amenable \ca\ $\mathcal{Z}$, now called the \emph{Jiang-Su algebra}, with the same Elliott invariant
as the complex numbers. This C*-algebra has the property that for any simple, separable, nuclear \ca\ $A$, and under a very mild assumption on $K_0(A)$, the algebras
$A$ and $A\otimes\mathcal{Z}$ have the same Elliott invariant (but need not be isomorphic, as
Tom's examples \cite{Tom_classification_2008} show). 

It became apparent that $\mathcal{Z}$-stability\footnote{We say that a \ca\ $A$ is 
\emph{$\mathcal{Z}$-stable}, or that it absorbs $\mathcal{Z}$ tensorially, if there is an isomorphism $A\otimes \mathcal{Z}\cong A$.}
needed to be added as an assumption in order to obtain classification by the Elliott invariant.
After a series of breakthroughs and hundreds of pages written on the matter, the revised
program came to a successful end: the last
step 
is a combination
of the works of 
Elliott-Gong-Lin-Niu \cite{EllGonLinNiu_classification_2015} 
and Tikuisis-White-Winter
\cite{TikWhiWin_quasidiagonality_2017}, and a
shorter proof recently appeared in \cite{CGSTW_classifying_2023}. The outcome is the following:

\begin{thm}(Classification Theorem).\label{cnj:Elliott}
Let $A$ and $B$ be simple, separable, unital,
nuclear, $\mathcal{Z}$-stable C*-algebras satisfying the UCT. 
Then $A$ and $B$ are isomorphic if and only
if their Elliott invariants are isomorphic.
\end{thm}

This classification result is
virtually optimal: except the possibly redundant requirement
on the UCT, all other conditions on the C*-algebras are needed 
in order to obtain
a classification result in terms of the Elliott invariant. 

\begin{df}
We will say that a C*-algebra is \emph{classifiable} if it satisfies the assumptions of the \autoref{cnj:Elliott}.
\end{df}

With such a powerful classification theorem at our disposal, 
it becomes an imperative task to identify interesting classes
of C*-algebras to which it can be applied. One
of the most natural families of C*-algebras arises from
topological dynamics via the (reduced) crossed product construction described in \autoref{df:FullRedCP}.

We now specify the class of actions we will be focusing on.

\begin{df}\label{df:DynSyst} By a \emph{(topological) dynamical system} we 
will always mean an action $G\curvearrowright X$ of a countable discrete group $G$, via homeomorphisms of a compact metrizable space $X$.
We may also sometimes say that $G\curvearrowright X$ is an 
\emph{action} of $G$ on $X$.
\end{df}

\begin{nota} Let $G\curvearrowright^\alpha X$ be a dynamical system. With a slight abuse of notation, we often also denote 
 by $\alpha$ the induced action of $G$ on $C(X)$, which is 
 given by $\alpha_g(a)(x)=a(g^{-1}\cdot x)$ for all $g\in G$, 
all $a\in C(X)$, and all $x\in X$.
\end{nota}

We explicitly exclude dynamical systems with a continuous
parameter (such as flows, or actions of $\mathbb{R}$), or actions 
by discrete uncountable groups, and 
we also exclude dynamical systems on non-compact spaces.
This is somewhat artificial, as there definitely exist interesting topological actions that do not satisfy our definition of a dynamical system, but our terminology is justified by \autoref{eg:CPunitalSep}.


In recent years, the following problem has attracted a great
deal of attention:

\begin{pbm}\label{pbm:DynCharClassif} Give a dynamical characterization of those
 dynamical systems $G\curvearrowright X$ for which $C(X)\rtimes_{\lambda} G$ satisfies the assumptions of \autoref{cnj:Elliott}.
\end{pbm}

In order to tackle \autoref{pbm:DynCharClassif}, we can go through its assumptions and try to find dynamical counterparts. 
Unitality and separability of $C(X)\rtimes_{\lambda} G$ are very easy to characterize, and are in fact automatic whenever we work with 
dynamical systems as defined in \autoref{df:DynSyst}.

\begin{ex}\label{eg:CPunitalSep}
Let $X$ be a locally compact, Hausdorff space, let $G$ be a discrete group, and let $G\curvearrowright X$ be an action by homeomorphisms\footnote{This is more general than what we called ``dynamical system'', since we do not require $X$ to be compact or metrizable, nor do we require $G$ to be countable.}.
\be\item $C(X)\rtimes_{\lambda} G$ is unital if and only if $X$ is compact.
\item $C(X)\rtimes_{\lambda} G$ is separable if and only if $X$ is second countable\footnote{For compact spaces, second countability is equivalent to metrizability.} and $G$ is countable.
\ee
\end{ex}

On the other hand, the remaining assumptions in \autoref{cnj:Elliott} are not automatic for dynamical systems.
We discuss each of them separately.

\section{Nuclearity of crossed products}
Nuclearity is not automatic for crossed products of
dynamical systems, as this already fails for group C*-algebras:
$C_{\lambda}^*(\mathbb{F}_2)$ is not nuclear; see also \autoref{eg:ConvNuclearCP}. On the other hand, we 
know precisely when it holds (see \autoref{thm:AmenIffNuc}): $C(X)\rtimes_{\lambda} G$ is nuclear if and only if $G\curvearrowright X$ is amenable, in the sense of the following definition. For a discrete group $G$, we write $P(G)\subseteq \ell^1(G)$ for the 
space of probability measures on $G$, and endow it with the 
canonical action given by $(g\cdot \mu)(E)=\mu(gE)$ for all $g\in G$, all $\mu\in P(G)$ and all $E\subseteq G$.

\begin{df}\label{df:AmenAc}
We say that a dynamical system $G\curvearrowright X$ is
\emph{amenable} if there exists a sequence $(\mu_n)_{n\in\N}$ of 
continuous functions $\mu_n\colon X\to P(G)$ satisfying
\[\lim_{n\to\infty}\max_{x\in X}\big\|\mu_n(g\cdot x)-g\cdot \mu_n(x)\big\|_1=0\]
for all $g\in G$.
\end{df}
 
\begin{eg}\label{eg:GamenActAmen} If $G$ is amenable, then any dynamical system
 $G\curvearrowright X$ is amenable. Indeed, if $(F_n)_{n\in\N}$ is a F{\o}lner sequence for $G$, one may take $\mu_n\colon X\to P(G)$ to be the constant function with value $\tfrac{1}{|F_n|}\sum_{g\in F_n}\delta_{g}$.
\end{eg}

There are many examples of amenable actions of nonamenable groups, for instance the canonical action of $\mathbb{F}_2$ on its Gromov's boundary; see part (4) of \autoref{egs:TopActions}. 

Since the measures that witness amenability of an action
of an amenable group do not depend on the space (see \autoref{eg:GamenActAmen}), it is perhaps not surprising that proving nuclearity of crossed products by amenable groups is easier. Accordingly, and since this case is often sufficient for most purposes, we leave the proof of this case as a guided 
exercise. Recall that a unital, separable \ca\ $A$ is \emph{nuclear} 
if there exist sequences $\varphi_n\colon A\to M_{k_n}$ and 
$\psi_n\colon M_{k_n}\to A$ of unital completely positive maps 
such that $\lim\limits_{n\to \infty}\|\psi_n(\varphi_n(a))-a\|=0$ for all $a\in A$.

\begin{ex}
Let $G$ be amenable and let $G\curvearrowright X$ be a dynamical system, and set $A=C(X)\rtimes_{\lambda} G$. 
The goal of this exercise is to show that $A$ is nuclear. 
Let $(F_n)_{n\in\N}$ be a F{\o}lner sequence. Set $k_n=
|F_n|$ and identify $M_{k_n}$ with $\B(\ell^2(F_n))$. Define 
maps $\varphi_n\colon A\to C(X)\otimes M_{k_n}$ and 
$\psi_n\colon C(X)\otimes M_{k_n}\to A$ by
\[ 
\varphi_n(au_s)=\sum_{g\in F_n\cap sF_n} \alpha_{g^{-1}}(a)\otimes e_{g,s^{-1}g}
\, \ \mbox{ and } \, \ \psi_n(a\otimes e_{g,h})=\tfrac{1}{|F_n|}\alpha_g(a) u_{gh^{-1}}.\]

\be 
\item Explain why these maps (especially $\varphi_n$) are well-defined.
\item Show that $\varphi_n$ and $\psi_n$ are unital and completely positive.
\item Show that $\lim\limits_{n\to \infty}\|\psi_n(\varphi_n(a))-a\|=0$ for all $a\in C(X)\rtimes_{\lambda} G$. 
\emph{Hint: do it first for $a=f\mathbbm{1}_g$ for some $f\in C(X)$ and some $g\in G$.}
\item Using that $C(X)$ is nuclear, show that the above implies that $C(X)\rtimes_{\lambda} G$ is nuclear.
\ee
\end{ex}

We now continue with the main treatment of this section.

\begin{df}
Let $G\curvearrowright X$ be a dynamical system.
A function $\xi\colon X \times G\to\C$ is said to be \emph{of positive type} if for
every $x \in X$, every positive integer $n$, every $g_1, \ldots, g_n \in G$, and every $z_1, \ldots, z_n \in \mathbb{C}$, we have:
\[
\sum_{j,k=1}^n \xi\big(g_j^{-1} \cdot x, g_j^{-1} g_k\big) \, z_j \overline{z_k} \geq 0.
\] 
\end{df}

We will need the following characterization of amenability in terms of functions of positive type; see \cite[Proposition~2.5]{Ana_amenability_2002} for a proof.

\begin{prop}\label{prop:CharAmen}
Let $G\curvearrowright X$ be a dynamical system. Then $G\curvearrowright X$ is amenable if and only if
there exists a net $(\xi_j)_{j\in J}$ of positive type functions in $c_c(G, C(X),\alpha)$ which converges to 1 uniformly on compact subsets of $X\times G$.
\end{prop}

In order to prove that crossed products of amenable actions are nuclear, we will use the notion of a $G$-$C(X)$-algebra, which we define next. We write $Z(B)$ for 
the center of a C*-algebra $B$.

\begin{df}
 Let $G\curvearrowright X$ be a dynamical system. A \emph{$G$-$C(X)$-algebra} is
a C*-algebra $A$, together with an action $\alpha\colon G\to\Aut(A)$, and a unital,
equivariant homomorphism $\rho\colon C(X)\to Z(M(A))$.
\end{df}

Finally, one more lemma, without proof, about
completely positive maps from crossed products to matrix algebras 
and how to replace them by nearby ones with finite support.

\begin{lma}\label{lma:CpcCP}
Let $G$ be a discrete group, and let $\Phi\colon C(X)\rtimes_{\lambda} G\to M_n$
be a completely positive map.
Let $\varepsilon > 0$, let $a \in C(X)$, and let a finite subset $F \subset G$ be given. Then there exists a completely positive map
\[
\widetilde{\Phi}\colon C(X)\rtimes_{\lambda} G \to M_n
\]
such that:
\begin{enumerate}
    \item $\| \Phi(\lambda_g) - \widetilde{\Phi}(\lambda_g) \| \leq \varepsilon$ for every $g \in F$;
    \item $g \mapsto \widetilde{\Phi}(f \lambda_g h)$ has finite support for every $f,h \in C(X)$.
\end{enumerate}
\end{lma}

We will use the following characterization
of nuclearity (see \cite[Theorem~3.8.7]{BroOza_algebras_2008}): a C*-algebra $A$ is nuclear if and only if for every C*-algebra $B$, the canonical map $A\otimes_{\mathrm{max}}B\to A\otimes_{\mathrm{min}}B$ is 
an isomorphism. In the upcoming proof will also use some tools not covered in these notes, such as Hilbert C*-modules and weak containment of representations; we refer the reader to 
\cite{BroOza_algebras_2008} for the necessary background.

The following is the expected result. We remind the reader that our acting groups
are assumed to be discrete; without this assumption, the implications 
(1) $\Rightarrow$ (2) $\Rightarrow$ (3) are true, but (3) $\Rightarrow$ (1) 
is false (already for group C*-algebras).

\begin{thm}\label{thm:AmenIffNuc}
Let $G\curvearrowright X$ be a dynamical system. Then the following are equivalent:
\be\item $G\curvearrowright X$ is amenable.
\item For every $G$-$C(X)$-algebra $A$, the canonical map $A\rtimes G\to A\rtimes_{\lambda} G$ is an isomorphism;
\item $C(X)\rtimes G$ is nuclear.
\ee
\end{thm}
\begin{proof} 
(1) $\Rightarrow$ (2). Without loss of generality, we may assume that $A$ is unital (see \autoref{ex:AssumeUnital}).
Let $(\varphi, u)$ be a covariant representation of $A$ on a Hilbert space $\mathcal{H}$; again without loss of generality, we may assume that $\varphi$ is unital. 
Recall the construction of the covariant representation $(\Hi^G,\lambda^\Hi,\varphi^G)$ from \autoref{df:RegCovRep}.
We will show that $\varphi\rtimes u$ is weakly contained in the representation $(\varphi^G,\lambda^\Hi)$ on the Hilbert space $\Hi^G=\ell^2(G,\mathcal{H})$.

Let $(\xi_j)_{j\in J}$ be a net as in \autoref{prop:CharAmen}. 
Given $\xi \in \mathcal{H}$, $x\in X$ and $g\in G$, define
\[
\eta_j(g) = u_{g^{-1}} \varphi(\xi_j(g)) \ \ \mbox{ and } \ \
k_j(x, g) = \sum_{g\in G} \xi_j(x, h) \, \xi_j(g^{-1} \cdot x, g^{-1} h).
\]

For $f \in C_c(X \times G)$, a straightforward computation gives:
\begin{align*}
\Big\langle \eta_j, (\varphi^G \times \lambda)(f) \eta_j \Big\rangle
&= \sum_{g\in G} \langle \eta_j(g), \varphi(\alpha_{g^{-1}}(f(h))) \eta_j(h^{-1}g) \rangle\\
&= \sum_{g\in G} \langle \varphi(k_j(h)) \xi, \varphi(f(h)) u_h \rangle.
\end{align*}

Since $(\xi_j)_{j\in J}$ converges to $1$ uniformly on compact subsets of $X \times G$, for every $\xi\in \Hi^G$ we deduce that
\[
\lim_{j\in J} \Big\langle \eta_j, (\varphi^G \times \lambda)(f) \eta_j \Big\rangle = \Big\langle \xi, (\varphi \rtimes u)(f) \xi \Big\rangle.
\]
Therefore, $\varphi \rtimes u$ is weakly contained in $\varphi^G \times \lambda^\Hi$, thus proving (2).

\vspace{.2cm}

(2) $\Rightarrow$ (3). Let $B$ be any $C^*$-algebra. Since $C(X) \otimes B$ is a $G$-$C(X)$-algebra in an obvious way, using (2) at the second step, we have:
\begin{align*}\label{eq:3.1}\tag{3.1}
(C(X)\rtimes_{\lambda} G) \otimes_{\min} B \cong (C(X)\otimes_{\min} B)\rtimes_{\lambda} G = (C(X) \otimes_{\min} B)\rtimes G.
\end{align*}
Moreover, 
\begin{align*}\label{eq:3.2}\tag{3.2}
C(X) \otimes_{\min} B \cong C(X) \otimes_{\max} B 
\end{align*}
because $C(X)$ is nuclear. Using at the third step that $C(X)\rtimes G\cong C(X)\rtimes_{\lambda} G$ by assumption, it follows that:
\begin{align*}
(C(X)\rtimes_{\lambda} G) \otimes_{\min} B &\stackrel{\eqref{eq:3.1}}{\cong} (C(X) \otimes_{\min} B)\rtimes G \\
&\stackrel{\eqref{eq:3.2}}{\cong} (C(X) \otimes_{\max} B)\rtimes G \\
& \ \, \cong \ (C(X)\rtimes G) \otimes_{\max} B \\
&\ \,  \cong \ (C(X)\rtimes_{\lambda} G) \otimes_{\max} B.
\end{align*}
Since all the isomorphisms involved are the canonical ones, we deduce
that $C(X)\rtimes_{\lambda} G$ is nuclear.

\vspace{.2cm}

(3) $\Rightarrow$ (1).
Given $\varepsilon > 0$ and a finite subset $F\subseteq G$, we will prove the existence of a compactly supported positive type function $h$ such that
\[
\sup_{(x, g) \in X \times F} |h(x, g) - 1| \leq \varepsilon;
\]
once we do this, the result will follow from \autoref{prop:CharAmen}.
Given $g \in G$, with a slight abuse of notation we denote by $\mathbbm{1}_g\colon X\times G\to \C$ by
\[
\mathbbm{1}_{g}(x, t) =
\begin{cases}
1, & \text{if } t = g, \\
0, & \text{otherwise}.
\end{cases}
\]
We view $\mathbbm{1}_g$ as an element of the Hilbert C*-module $\ell^2(G) \otimes C(X)$ in the obvious way.
Using that $C(X)\rtimes_{\lambda} G$ is nuclear, let
\[
\Phi\colon C(X) \rtimes_{\lambda} G \to M_n(\mathbb{C}) \quad \text{and} \quad \Psi \colon M_n(\mathbb{C}) \to C(X) \rtimes_{\lambda} G
\]
be completely positive contractions, such that if we set 
$\Theta=\Psi\circ \Phi\colon C(X) \rtimes_{\lambda} G \to C(X) \rtimes_{\lambda} G$,
then $\|\Theta(\lambda_g)-\lambda_g\|<\frac{\ep}{2}$ for all $g\in F$.
Define
\[
\xi(x, g) := \big\langle \mathbbm{1}_g, \Theta(\lambda_g) \mathbbm{1}_{1}\big\rangle(g^{-1} \cdot x).
\]
It is straightforward to show that
\[
\xi(g^{-1} \cdot x, g^{-1} h) = \big\langle \mathbbm{1}_{g^{-1}}, \Theta( \lambda_{g^{-1} h}^*) \mathbbm{1}_{h^{-1}}(x) \big\rangle.
\]
We deduce that $\xi$ is a function of positive type since $\Theta$ is completely positive.
Moreover, we have
\[
\sup_{x \in X} \left| \xi(x, g) - \big\langle \mathbbm{1}_g, \lambda_g \mathbbm{1}_{1}(g^{-1} \cdot x) \big\rangle \right| \leq \frac{\varepsilon}{2}.
\]
On the other hand, for $g\in F$ we have
\[
\langle \mathbbm{1}_g, \lambda_g \mathbbm{1}_{1}(g^{-1} \cdot x) \rangle = 1.
\]

We are not done with the proof, because $\xi$ need not have compact support; we will fix this using \autoref{lma:CpcCP}. Apply said lemma to choose a completely positive map
$\widetilde{\Phi} \colon C(X) \rtimes_{\lambda} G \to M_n$
such that
\[
\|\Phi(\lambda_g) - \widetilde{\Phi}(\lambda_g)\| \leq \frac{\varepsilon}{2}
\]
for every $g \in F$, and such that $g \mapsto \widetilde{\Phi}( \lambda_g)$ has finite support. We now define
$\widetilde{\Theta}= \Psi \circ \widetilde{\Phi}$ and set
\[
\widetilde{\xi}(x, g) := \Big\langle \mathbbm{1}_g, \Phi(\lambda_g) \mathbbm{1}_{1}(g^{-1} \cdot x) \Big\rangle.
\]

Then $\widetilde{\xi}$, when regarded as a function $G\to C(X)$, 
has finite support. Moreover, it is of positive type because $\widetilde{\Phi}$ is completely positive, and a routine calculation using the triangle inequality shows that 
\[
\sup_{x \in X} \left| \widetilde{\xi}(x, g) - \big\langle \mathbbm{1}_g, \lambda_g \mathbbm{1}_{1}(g^{-1} \cdot x) \big\rangle \right| \leq \ep.
\]
We conclude from \autoref{prop:CharAmen} that $G\curvearrowright X$ is amenable, as desired.
\end{proof}

\begin{ex}\label{ex:AssumeUnital}
In the proof of (1) $\Rightarrow$ (2) above, show that one can 
indeed assume that $A$ is unital.
\end{ex}

As it turns out, by the work of Tu \cite{Tu_conjecture_1999}, the UCT is automatic in the cases we are interested in.

\begin{thm} (Tu).
Let $G\curvearrowright X$ be an amenable dynamical system.
Then $C(X)\rtimes G$ satisfies the UCT.
\end{thm}

\section{Simplicity of crossed products}\label{sec:Simple}

Before turning to simplicity,  we have a look at the ideal structure of crossed products.
A complete description of the ideals in a crossed product is at the moment far out of reach, since it combines in a subtle and 
still mysterious way the dynamics of the action and the internal
structure of the group.
On the other hand, it is possible to describe some families of ideals in crossed products. 

Certain ideals are induced purely dynamically.

\begin{prop}\label{prop:OpenGinv}
Let $G\curvearrowright X$ be a dynamical system, and let $U\subseteq X$ be an open, $G$-invariant set. Then $I_U:=C_0(U)\rtimes_{\lambda} G$ is an ideal in $C(X)\rtimes_{\lambda} G$.
\end{prop}
\begin{proof}
Note that the restriction of a regular covariant 
representation of $(G, C(X),\alpha)$ to $(G, C_0(U),\alpha)$ is again a regular covariant representation.
Using this, it is easy to see that $I_U$ is a subalgebra of $C(X)\rtimes_{\lambda} G$. In order to check that it is 
an ideal, it suffices to check that $c_c(G,C_0(U),\alpha)$ is an ideal in $c_c(G,C(X),\alpha)$. To see this, let 
$a\in C(X)$, let $b\in C_0(U)$, and let $g,h\in G$.
Then
\[(a\mathbbm{1}_g)\ast (b\mathbbm{1}_h)=a\alpha_g(b)\mathbbm{1}_{gh}\]
belongs to $c_c(G, C_0(U),\alpha)$ since $\alpha_g(b)$ belongs to $C_0(U)$. Similarly, one shows that
$(b\mathbbm{1}_h)\ast (a\mathbbm{1}_g)$ belongs to $c_c(G,C_0(U),\alpha)$, as desired.
\end{proof}

With similar arguments, one can show the following.

\begin{ex}
Let $G\curvearrowright X$ be a dynamical system.
If $I$ is an ideal in $C(X)\rtimes G$, show that $I\cap C(X)$ has the form $C_0(U_I)$ for some open, $G$-invariant set $U_I\subseteq X$. Moreover,
$U_{I_U}=U$ and $I\subseteq I_{U_I}$. Is it true that $I=
I_{U_I}$?
\end{ex}

Recall that a dynamical system $G\curvearrowright X$ is
\emph{minimal} if there are no nontrivial $G$-invariant, closed (equivalently, open) subsets of $X$. The following 
is an immediate consequence of \autoref{prop:OpenGinv}.

\begin{cor}\label{cor:SimpleMinimal}
Let $G\curvearrowright X$ be a dynamical system. If $C(X)\rtimes_{\lambda} G$ is simple, then $G\curvearrowright X$ is minimal. 
\end{cor}

In other words, minimality is a necessary condition for 
simplicity. It is, however, not sufficient: the trivial action $\Z\curvearrowright \{\ast\}$ is minimal, but its crossed product $\C\rtimes\Z\cong C(S^1)$ is not simple.
The reason why the trivial action of $\Z$ gives rise to a 
non-simple crossed product is a combination of the fact that $\Z$ is amenable and that the trivial action is not (topologically) free. 



In full generality, it is not known in dynamical terms
when a dynamical system $G\curvearrowright X$ gives rise
to a simple (reduced) crossed product (see \autoref{qst:Ozawa}). 
However, when the action is amenable, we do have a 
complete answer. 

We need a preliminary result, which is left as an exercise.
 
\begin{ex}\label{ex:Ex}
Let $G\curvearrowright X$ be an amenable action,
let $H\leq G$ be a subgroup, and let $U\subseteq X$ 
be an open set. If the restricted action $H\curvearrowright U$ is trivial, then $H$ is amenable.
\emph{(Hint: prove the statement first in the case that $H=G$ and $U=X$, and then show that amenability of dynamical systems passes to subgroups, and also to restrictions to invariant open subsets.)}
\end{ex}

Next, we recall some standard definitions.

\begin{df}
A $G\curvearrowright X$ dynamical system is said to be
\be\item \emph{free}, if for every $g\in G\setminus\{1\}$, the fixed point set $\{x\in X\colon g\cdot x=x\}$ is empty. 
\item \emph{topologically free}, if for every $g\in G\setminus\{1\}$, the fixed point set $\{x\in X\colon g\cdot x=x\}$ has empty interior.
\ee
\end{df}

The following equivalent formulation for topological freeness for actions of \emph{countable} groups is left as an exercise:

\begin{ex}
A dynamical system $G\curvearrowright X$ is topologically free 
if and only if there exists a dense, $G$-invariant 
subset $Y\subseteq X$ such that the action of $G$ on $Y$ is free.
\end{ex}

Topological freeness gives rise a certain \emph{pinching} property in the reduced crossed product; see \autoref{prop:Pinching}, which is at the center of our
arguments to show simplicity.
In preparation for it, we need a simple lemma.
For $g\in G$
we will write $X_g=\{x\in X\colon g\cdot x=x\}$
for the fixed point set of $g$.

\begin{lma}\label{lma:topfreeSeparation}
Let $G\curvearrowright X$ be a dynamical system, let $g\in G\setminus\{1\}$, let $h\in C(X)$, and let $x_0\notin X_g$. Then there exists $f_g\in C(X)_+$ of 
norm one with $f_g(x_0)=1$ such that, in the crossed product $C(X)\rtimes G$, we have
\[f_ghu_gf_g=0.\]
\end{lma}
\begin{proof}
Let $V_1$ and $V_2$ be disjoint open sets such that $x_0 \in V_1$ and $g^{-1}\cdot x_0 \in V_2$. Set
\[
V = V_1 \cap (g\cdot V_2).
\]
Then $x_0 \in V \subseteq V_1$ and $g^{-1}\cdot V \subseteq V_2$, from which it follows that
\[
g^{-1}\cdot V \cap V = \emptyset.
\]

Let $f_g \in C(X)$ with $f_g(x_0)=1$ satisfy $0 \leq f_g \leq 1$ and $f_g(X \setminus V) = 0$.
Then
\[
f_g h u_g f_g = u_g \alpha_{g^{-1}}(f_g h) f_g
\]
vanishes because the support of $\alpha_{g^{-1}}(f_g h)$ is contained in $g^{-1}\cdot V$, and the support of $f_g$ is contained in $V$. This finishes the proof.
\end{proof}

Recall the definition of the conditional expectation $E\colon C(X)\rtimes_{\lambda} G\to C(X)$, for $g\in G$, given in \autoref{df:CondExpCP} The following is 
\cite[Proposition~2.4]{ExeLacQui_partial_2022}.

\begin{lma}\label{prop:Pinching}
Let $G\curvearrowright X$ be a topologically free dynamical system. Let $a\in C(X)\rtimes_{\lambda} G$ be positive and let $\ep>0$. Then there exists $f\in C(X)_+$ with $\|f\|=1$ such that
\[\|E(a)\| \leq \|fE(a)f\| + \ep \ \ \mbox{ and } \ \  \|fE(a)f - faf\|\leq \ep.\]
\end{lma}
\begin{proof}
Since $E$ is contractive, it suffices to prove the lemma assuming that $a$ has the form $a = \sum_{g \in F} a_g u_g$, for a finite subset $F\subseteq G$ and coefficients $a_g\in C(X)$, for $g\in G$. Note that $E(a)=a_1$. Set
\[
V = \{x \in X \colon |a_1(x)| > \|a_1\| - \varepsilon\},
\]
which is open and nonempty.
Since the finite union $\bigcup_{g\in F\setminus\{1\}}X_g$ has empty interior, there exists $x_0 \in V$ such that $x_0 \notin X_g$ for every $g \in F \setminus \{1\}$. and by \autoref{lma:topfreeSeparation} there exist functions $0\leq f_g\leq 1$ satisfying:
\[
f_g(x_0) = 1 \ \ \mbox{ and } \ \ f_g a_g u_g f_g=0.
\]

Set $f= \prod_{g \in F \setminus \{1\}} f_g\in C(X)$, and 
note that $f$ is positive and has norm one.
Since $x_0 \in V$, we have 
\[\|f a_1 f\| \geq |a_1(x_0)| > \|a_1\| - \varepsilon,\]
which gives the first inequality in the statement.

For the second one, we have:
\begin{align*}
f a_1 f - f a f &= \sum_{g \in F \setminus \{1\}} f a_g u_g f =0 ,
\end{align*}
as desired. 
\end{proof}

We are now ready to characterize simplicity of crossed
products of amenable dynamical systems. The following is \cite[Theorem~2]{ArcSpi_topologically_1994}; note that what they call ``regular" (full and reduced crossed products agree canonically) is a consequence of amenability by \autoref{thm:AmenIffNuc}.

\begin{thm}\label{thm:SimpleCP}
Let $G\curvearrowright X$ be a dynamical system.
If $G\curvearrowright X$ is minimal and topologically free, then $C(X)\rtimes_{\lambda} G$ is simple. The converse is true if $G\curvearrowright X$ is amenable.
\end{thm}
\begin{proof} 
Assume that $G\curvearrowright X$ is minimal and topologically free. Let $I$ be a proper ideal in $C(X)\rtimes_{\lambda} G$; we claim that $I=\{0\}$. 
The intersection of $I$ with $C(X)$ must be a proper $G$-invariant ideal in $C(X)$, corresponding to a proper $G$-invariant open set, and since $G\curvearrowright X$ is minimal, we must have $I\cap C(X)=\{0\}$. 

Assume by contradiction that $I\neq\{0\}$, and choose
a nonzero positive element $a\in I$. Fix $\ep>0$; we will show that $\|E(a)\|<2\ep$. Once we do this, it will follow from \autoref{thm:CondExpFaithful} that $a=0$. 

Let $\pi_I\colon C(X)\rtimes_{\lambda} G \to \big(C(X)\rtimes_{\lambda} G\big)/I$ denote the canonical quotient map, 
and note that the restriction of $\pi_I$ to $C(X)$ is injective because $I\cap C(X)=\{0\}$.
Let $f \in C(X)_+$ be as in the conclusion of \autoref{prop:Pinching}. Using said lemma at the first and third steps, using that $\pi_I$ is isometric on $C(X)$ at the second step, and using that $faf$ belongs to $I$ at the last step, we get
\begin{align*}
\|E(a)\| &\leq \|f E(a) f\| + \varepsilon \\
&= \|\pi_I(f E(a) f)\| + \varepsilon \\
&\leq \|\pi_I(f a f)\| + 2\varepsilon =2\varepsilon,
\end{align*}
as desired. We conclude that $I=\{0\}$ and hence $C(X)\rtimes_{\lambda} G$ is simple.

\vspace{.2cm}

Let us assume now that $G\curvearrowright X$ is amenable
and that $C(X)\rtimes_{\lambda} G$ is simple. Minimality of
$G\curvearrowright X$ follows from \autoref{cor:SimpleMinimal}, so it 
remains to show that it is topologically free.
Fix $x\in X$, and let 
\[G_x^\circ=\{g\in G\colon \mbox{there is an open set } x\in U\subseteq X \mbox{ where $g$ acts trivially}\}\] 
denote the open stabilizer of $x$. We will show that $G_x^\circ=\{1\}$; note that $G_x^{\circ}$ is amenable by \autoref{ex:Ex}. We will use that $C(X)\rtimes G=
C(X)\rtimes_{\lambda} G$ by \autoref{thm:AmenGFullRedCP}.

Let $g \in G_x^\circ$ and choose an open set $U \subseteq X$ containing $x$ on which $g$ acts as the identity.
Let $\pi\colon C(X) \rtimes G\to \mathcal{B}(\ell^2(G/G_x^\circ))$ be the representation given by
\[
\pi(f u_s)(\mathbbm{1}_{pG_x^\circ}) = f(spx) \mathbbm{1}_{spG_x^\circ}
\]
for all $f \in C(X)$, all $s \in G$, and all $p \in G$.
Let $f \in C(X)$ be a nonzero function supported on $U$. Then $f(px) \neq 0$ implies $px \in U$,
and hence $gpG_x^\circ = pG_x^\circ$ in $G/G_x^\circ$. Thus, $\pi((1 - u_g)f) = 0$. Since $\pi$ is injective (because $C(X)\rtimes G$ is simple), one has $g = 1$.
This implies $G_x^\circ = \{1\}$, as desired. \end{proof}

\begin{ex} Check that the map $\pi$ defined in the previous proof is really a continuous representation
 of $C(X)\rtimes G$. \emph{(Hint: amenability of $G_x^\circ$ is needed here, specifically that $C^*(G_x^\circ)=C^*_{\lambda}(G_x^\circ)$.)}
\end{ex}

\begin{eg}\label{eg:ConvNuclearCP} The converse above is not true if $G\curvearrowright X$
 is not assumed to be amenable: simplicity of $C(X)\rtimes G$ always implies minimality of $G\curvearrowright X$ by \autoref{prop:OpenGinv}, but it may not imply any form of freeness outside of the amenable case. For instance, $C^*_{\lambda}(\mathbb{F}_2)$, which is the crossed product of 
 the trivial action of $\mathbb{F}_2$ on a point, is simple
 although $G\curvearrowright \ast$ is not topologically free.
\end{eg}

Topological freeness can sometimes be upgraded to freeness.

\begin{prop}\label{prop:GabelTopFree}
Let $G$ be an abelian group and let $G\curvearrowright X$
be a minimal dynamical system. Then $G\curvearrowright X$ is
topologically free if and only if it is free.
\end{prop}
\begin{proof} 
One direction is always true: freeness always implies topological freeness. Conversely,
let $g\in G$ and let $x\in X$ satisfy $g\cdot x=x$. We claim that $g$ fixes the entire $G$-orbit of $x$. To see this, let $h\in G$. Using that $G$ is abelian, we get 
\[g\cdot (h\cdot x)=h\cdot (g\cdot x)=h\cdot x,\]
as desired. By continuity, $g$ must fix the closure of $G\cdot x$, which equals $X$ by minimality. Since $G\curvearrowright X$ is topologically free, we conclude that $g=1$ and $G\curvearrowright X$ is free.  
\end{proof}

Not surprisingly, under the stronger assumption of \emph{freeness}, showing simplicity becomes easier. 
We leave the proof of this particular case as a guided exercise.

\begin{ex}\label{ex:Free}
Let $G\curvearrowright X$ be a free, minimal action. The goal of this exercise is to show that $C(X)\rtimes_{\lambda} G$ is simple.
\be\item Let $F\subseteq G\setminus\{e\}$ be a finite set.
Show that there exist $n\in \N$ and $s_1,\ldots,s_n\colon X\to S^1$ such that 
\[
\sum_{j=1}^n s_k(x)\overline{s_k(g^{-1}\cdot x)}=0
\]
for all $x\in X$ and all $g\in F$. \emph{(Hint: Given $x_0\in X$, 
do it first for $x$ in some neighborhood of $x_0$. Then use a compactness argument.}
\item Using the notation from the previous part, deduce that 
$\sum_{k=1}^ns_ku_gs_k^*=0$ in $C(X)\rtimes_{\lambda} G$, for every $g\in F$.
\ee
Let $I$ be a proper ideal in $C(X)\rtimes_{\lambda} G$. 
\be\item[(3)] Show that $I\cap C(X)=\{0\}$. \emph{(Hint: $I\cap C(X)$ is a proper $G$-invariant ideal in $C(X)$.)}
\item[(4)] Show that $E(I)\subseteq I$. \emph{(Hint: let 
$a\in I$ and $\ep>0$ be given. Find $b\in c_c(G,C(X),\alpha)$ with $\|a-b\|<\ep$. With $F=\supp(b)\setminus\{e\}$, let $s_1,\ldots,s_n$ be as in part (1) of the exercise. Check that $E(b)=\sum_{k=1}^n s_kbs_k^*$ and use this to show that $E(a)$ is within $\varepsilon$ of an element in $I$.)}
\item[(5)] Conclude that $I=\{0\}$. \emph{(Hint: Use that $E$ is faithful on $C(X)\rtimes_{\lambda} G$.)}
\ee
\end{ex}

\section{$\mathcal{Z}$-stability of crossed products}
We conclude that the classifiability problem for crossed products of dynamical systems boils down to the following 
question.

\begin{qst}\label{qst:Main}
Let $G\curvearrowright X$ be an amenable, minimal and 
topologically free dynamical system. When does $C(X)\rtimes G$ 
absorb $\mathcal{Z}$ tensorially?
\end{qst}

The answer to the above question is not ``always", and Giol-Kerr produced 
in \cite{GioKer_subshifts_2010} an example of a free, minimal homeomorphism $\alpha$ of a compact
metric space $X$ such that $C(X)\rtimes_{\alpha} \Z$ is not $\mathcal{Z}$-stable (although it is simple, separable, unital, nuclear and satisfies the UCT). Thus, $\mathcal{Z}$-stability is in general not automatic for simple, nuclear crossed products. 
 
We have not defined the Jiang-Su algebra, as it is indeed
rather technical and not particularly illuminating to see its construction. However, it is easier to characterize 
when a simple, nuclear C*-algebra is $\mathcal{Z}$-stable
than it is to define the Jiang-Su algebra itself. For this, we need a few notions. The following one is \cite[Definition~1.3]{WinZac_completely_2009}.

\begin{df}\label{df:cpoz} (Winter-Zacharias).
Let $\varphi\colon A\to B$ be a completely positive map 
between C*-algebras. We say that $\varphi$ has 
\emph{order zero} if $\varphi(a)\varphi(b)=0$ whenever
$a,b\in A$ satisfy $ab=0$.
\end{df}

It is easy to check that a homomorphism between \ca s automatically has order zero. The following gives a more general
class of examples that covers the kind of order zero maps we will be interested in (which appear in the proof of \autoref{thm:afZstab}). We also note that, by \cite[Theorem~2.3]{WinZac_completely_2009}, all ordero zero maps arise essentially in this way.

\begin{ex} Let $\psi\colon A\to B$ be a homomorphism between 
 \ca s, and let $h\in B_+$ be a contraction satisfying
 $h\psi(a)=\psi(a)h$ for all $a\in A$. Define a map 
 $\varphi\colon A\to B$ by $\varphi(a)=\psi(a)h$ for all $a\in A$. Show that $\varphi$ is completely positive and has order zero.
\end{ex}

The next notion can be traced back to the work of Cuntz 
on dimension functions in C*-algebras; see \cite{Cun_dimension_1978}.

\begin{df}\label{df:CtzComp}
Let $A$ be a \ca, and let $a,b\in A$ be positive
elements. We say that $a$ is \emph{Cuntz subequivalent} to $b$ in $A$, written $a\precsim_A b$ or just $a\precsim b$ if the algebra $A$ is clear from the context, if there exists a sequence $(c_n)_{n\in\N}$ in $A$ such that
$c_n^*bc_n\to a$ in the norm of $A$.
\end{df}

The verification that $\precsim$ is transitive is left as an 
exercise.

\begin{ex}\label{ex:CtzCompTransitive}
Let $A$ be a \ca, and let $a,b,c\in A$ be positive elements.
Show that $a\precsim b$ and $b\precsim c$ imply $a\precsim c$.
\end{ex}

For future use, we leave as an exercise the following description
of Cuntz subequivalence in commutative C*-algebras.

\begin{ex}\label{ex:CtzCompCX}
Let $X$ be a locally compact Hausdorff space and let $f,g\in C_0(X)$ be positive functions. Show that $f\precsim g$ if and only if 
\[\{x\in X\colon f(x)>0\}\subseteq \{x\in X\colon g(x)>0\}.\] 
\end{ex}

We are now ready to give a manageable characterization of 
$\mathcal{Z}$-stability, which is due to Hirshberg-Orovitz \cite{HirOro_tracially_2013}.

\begin{thm}\label{thm:ZstTrZst} (Hirshberg-Orovitz).
Let $A$ be a simple, separable, unital, nuclear C*-algebra.
Then $A$ is $\mathcal{Z}$-stable if and only if for every $\varepsilon>0$, every finite subset $F\subseteq A$, every $n\in\N$ and
every positive element $f\in A$, there exists a 
completely positive contractive order zero map $\varphi\colon M_n\to A$ such that
\be\item $\|\varphi(d)a-a\varphi(d)\|<\|d\|\varepsilon$ for all $a\in F$ and all $d\in M_n$; and
\item $1-\varphi(1)\precsim f$.\ee
\end{thm}

\begin{rem}\label{rem:AssumefCX} The notation $f\in A$ above is chosen because,
 when $A=C(X)\rtimes G$, we will show later that one can
 assume that $f\in C(X)$; see \autoref{prop:TopFreeCtzSubeq}.
\end{rem}

The above criterion for $\mathcal{Z}$-stability should be compared to a similar characterization of absorption of UHF-algebras of infinite type. We state it for the CAR algebra $M_{2^\infty}$ for the sake of simplicity.

\begin{thm}
Let $A$ be a separable, unital, nuclear C*-algebra.
Then $A$ is $M_{2^\infty}$-stable if and only if for every $\varepsilon>0$, every finite subset $F\subseteq A$, there exists a 
unital homomorphism $\varphi\colon M_2\to A$ such that
for all $d\in M_2$ we have
\[\|\varphi(d)a-a\varphi(d)\|<\|d\|\varepsilon.\]
\end{thm}

In order to make use of \autoref{thm:ZstTrZst}, we 
would like to simplify condition (2) about Cuntz
comparison in $C(X)\rtimes G$. We need a few elementary
facts about Cuntz comparison, whose proofs can be found
in \cite[Section~2]{GarPer_modern_2025}. For positive elements $a$ and $b$ in a C*-algebra $A$, we write $a\sim b$ whenever
$a\precsim b$ and $b\precsim a$. Recall the definition of $(a-\ep)_+$ from \autoref{df:CutDown}.

\begin{prop}\label{prop:CuComp}
Let $A$ be a \ca\ and let $\ep>0$.

\be
\item Let $a,b\in A_+$ satisfy $\|a - b\| < \ep$. Then 
$(a-\ep)_+\precsim b$.
\item Let $c\in A$. Then $(cc^*-\ep)_+\sim (c^*c-\ep)_+$.
\item Let $a,b\in A$ satisfy $0\leq a\leq b$. Then $(a-\varepsilon)_+\precsim (b-\varepsilon)_+$.
\ee
\end{prop}

Given $\ep>0$, let $f_{\ep}\colon [0,\infty)\to [0,\infty)$ be given
by $f_{\ep}(t)=\max\{t-\ep,0\}$.

\begin{df} \label{df:CutDown}
	Given a positive element 
	$a\in A$, we write
	$(a-\ep)_+$ for $f_{\ep}(a)$. The element $(a-\ep)_+$ is usually
	referred to as the \emph{$\ep$ cut-down} of $a$. \end{df}

\begin{ex}\label{ex:CutDownSubeq}
	Let $a\in A$ be a positive element and let $\ep>0$. Show 
	that $(a-\ep)_+\precsim a$. \emph{(Hint: this subequivalence holds in the commutative C*-algebra $C^*(a)$.)}
\end{ex}

Since $(a-\varepsilon)_+\leq a$ for any $\varepsilon>0$, we deduce from part~(3) of \autoref{prop:CuComp} that $(a-\varepsilon)_+\precsim a$.

Our next proposition is a particular case of 
\cite[Lemma~7.9]{Phi_large_2014}, and it asserts that
every positive element in the crossed product of a topologically free action dominates, in the Cuntz sense, a positive element in $C(X)$. In practice, this proposition will allow us to check 
certain ``smallness'' conditions inside $C(X)$ directly; 
see \autoref{rem:AssumefCX}.

\begin{prop}\label{prop:TopFreeCtzSubeq}
Let $G\curvearrowright X$ be a topologically free dynamical
system and let $a\in C(X)\rtimes G$ be a nonzero positive element.
Then there exists $h\in C(X)$ nonzero such that $h\precsim a$.
\end{prop}
\begin{proof} 
Without loss of generality, we can assume
that $\|a\| \leq 1$. Since the conditional expectation
$E$ is faithful (see \autoref{thm:CondExpFaithful}), we
have that $E(a) \in C(X)$ is a nonzero positive element. Set $\varepsilon = \tfrac{1}{6} \|E(a)\|$.
Choose $c \in c_c(G,C(X),\alpha)$ such that $\|c - a^{1/2}\| < \varepsilon$ and $\|c\| \leq 1$. Then
\be\item[(i)] $\|c c^* - a\| < 2\varepsilon$ and
\item[(ii)] $\|c^* c - a\| < 2\varepsilon$.
\ee

Use \autoref{prop:Pinching}, with $c^*c$ in place of $a$,
to find $f\in C(X)_+$ with $\|f\|=1$ such that
\be\item[(iii)]$\|E(c^*c)\| \leq \|fE(c^*c)f\| + \ep$ and
\item[(iv)] $\|fE(c^*c)f - fc^*cf\|\leq \ep$.
\ee Then
\[
\|f E(c^* c) f\| \stackrel{(iii)}{\geq} 
\|E(c^* c)\| - \varepsilon \stackrel{(ii)}{>} \|E(a)\| - 3\varepsilon = 3\varepsilon.
\]
Therefore, the element
\[
h:=(f E(c^* c) f - 3\varepsilon)_+
\]
belongs to $C(X)$ and is nonzero. 
In the following inequalities, we use 
part (1) of \autoref{prop:CuComp} at the second step;
part (2) of \autoref{prop:CuComp} at the third step;
part (3) of \autoref{prop:CuComp} for the inequality
$0\leq cf^2c^*\leq  cc^*$ at the fourth step; and 
again part (1) of \autoref{prop:CuComp} at the last step, to get
\[
h =(f E(c^* c) f - 3\varepsilon)_+\stackrel{(iv)}{\precsim} (fc^*cf-2\ep)_+
\sim (c f^2 c^* - 2\varepsilon)_+ \precsim (c c^* - 2\varepsilon)_+ \stackrel{(i)}{\precsim} a.
\]
The result then follows from transitivity of $\precsim$; see \autoref{ex:CtzCompTransitive}.
\end{proof}

A combination
of very deep results of Cuntz \cite{Cun_dimension_1978}, Blackadar-Handelman \cite{BlaHan_dimension_1982}, and
Haagerup \cite{Haa_quasitraces_2014}, gives the following
result.

\begin{thm}\label{thm:ExistTraces} (Cuntz, Blackadar-Handelman, Haagerup). Let $A$ be a simple, unital, nuclear\footnote{It is also enough to assume that $A$ is exact, but nuclearity is enough for our purposes.} C*-algebra. Then 
 $A$ is stably finite if and only if it admits a tracial state.
\end{thm}
 
The techniques to address \autoref{qst:Main} differ significantly depending on whether $G$ is amenable or not. This is not surprising, given the following.

\begin{prop}\label{prop:GamenIFFinvmeas}
Let $G\curvearrowright X$ be an amenable, minimal, topologically free dynamical system. Then the following
are equivalent:
\be\item $C(X)\rtimes G$ is stably finite;
\item There is a $G$-invariant probability measure on $X$; 
\item $G$ is amenable.
\ee
\end{prop}
\begin{proof} 
By \autoref{thm:ExistTraces}, (1) is equivalent to $C(X)\rtimes G$ admitting a
tracial state, which by \autoref{prop:TracesCP} is equivalent to (2). 
To show that (2) implies (3), fix $\nu\in \mathrm{Prob}_G(X)$.
Let $(\mu_n)_{n\in\N}$ be a sequence as in \autoref{df:AmenAc}, witnessing the fact that $G\curvearrowright X$ is amenable.
For $n\in\N$, define 
$\rho_n\in \mathrm{Prob}(G)$ by
\[\rho_n(E)=\int_X\mu_n(x)(E)\ d\nu(x)\]
for all $E\subseteq G$.
Then $\|\rho_n-g\cdot \rho_n\|_1\xrightarrow{n\to \infty}0$ for all $g\in G$,
and thus $G$ is amenable. 

Finally, we show that (3) implies (2). Fix a $G$-invariant mean $\phi\colon \ell^{\infty}(G)\to \C$ and let $\eta$ be any probability measure on $X$. Define $P_{\eta}\colon C(X)\to \ell^{\infty}(G)$ by
\[P_{\eta}(f)(g)=\int_X f(g\cdot x)\ d\eta(x)\]
for $f\in C(X)$ and $g\in G$, and note that $P_{\eta}$
is a unital positive $G$-equivariant map. 
Then $\phi\circ P_{\eta}\colon C(X)\to \C$ is a $G$-invariant state giving rise to a $G$-invariant probability measure on $X$. 
\end{proof}

\section{Traces of crossed products}\label{sec:Traces}

Before closing this chapter, we briefly discuss traces on
crossed products by dynamical systems. 
A complete and somewhat explicit description (in terms of positive-definite functions) was 
obtained by Ursu in \cite{Urs_characterizing_2021}, generalizing previously known cases by Kawamura, Takemoto, and Tomiyama \cite{KawTakTom_state_1990}. Here, we are interested in computing the trace space
of a classifiable crossed product, so we will mostly consider minimal, topologically free and amenable dynamical systems. 

Just like for ideals, some 
traces are induced purely dynamically. We denote by $P(X)$ the 
space of Borel probability measures on $X$, and by $\mathrm{Prob}_G(X)$
the subset of invariant measures.

\begin{prop}\label{prop:TracesCP}
Let $G\curvearrowright X$ be a dynamical system, and let $\mu\in \mathrm{Prob}_G(X)$. Then the map $\tau_{\mu} \colon C(X)\rtimes_{\lambda} G\to \C$ given by
\[\tau_{\mu}(au_g)=\begin{cases}
\int_Xa(x)\, d\mu(x), &\mbox{ if } g=e;\\
0, &\mbox{ else},
                 \end{cases}\]
is a trace. 

Conversely, if $\tau$ is a trace on $C(X)\rtimes_{\lambda} G$, 
then its restriction to $C(X)$ is given by integration against a measure $\mu_{\tau}\in \mathrm{Prob}_G(X)$.
\end{prop}

It is natural to ask whether all traces on crossed products have the form described in the proposition
above, so that there is a canonical affine homeomorphism $T(C(X)\rtimes G)\cong \mathrm{Prob}_G(X)$.
In general the answer is no, but there is a
dynamical characterization of this fact; see \autoref{prop:TracesEssFree}. 

For the dynamical systems we are interested in, the following gives a computation of $T(C(X)\rtimes G)$.
In its statement, we will regard functions in $C(X)$
as elements in $L^\infty(X,\mu)$ in the obvious way, 
and we will denote by
$\alpha$ the extension of the action on $C(X)$ to $L^\infty(X,\mu)$.

\begin{thm}\label{thm:Ursu} (\cite[Theorem~1.2]{Urs_characterizing_2021}). 
Let $G\curvearrowright X$ be a minimal, topologically free, amenable dynamical
system, and let $\mu\in \mathrm{Prob}_G(X)$. Then there is a natural
one-to-one correspondence between the sets
\[\{\tau\in T\big(C(X)\rtimes G\big)\colon \tau|_{C(X)}=\tau_{\mu}\}\]
and the set of all functions $a\colon G\to L^\infty(X,\mu)$ satisfying:
\be\item[(i)] $a_1 = 1$;
\item[(ii)] $a_g \xi = a_g\alpha_g(\xi)$ for all $\xi\in L^\infty(X,\mu)$ and all $g\in G$;
\item[(iii)] $\alpha_h(a_g)= a_{hgh^{-1}}$ for all $g,h\in G$;
\item[(iv)] The matrix $[a_{hg^{-1}}]_{g,h\in F}$ is positive for all finite subsets $F\subseteq G$.\ee
\end{thm}

We will not prove the above theorem, but one part of the 
statement is easy to show.
Indeed, the way in which a function $a\colon G\to L^\infty(X,\mu)$ gives rise (uniquely) to a trace on the crossed product is described in the following exercise.

\begin{ex}
Adopt the notation and assumptions of \autoref{thm:Ursu}.
\be \item
Given a function $a\colon G\to L^\infty(X,\mu)$ satisfying (i)-(iv) above, show that there is a trace 
on $C(X)\rtimes G$, extending $\tau_{\mu}|_{C(X)}$, which is given by $\tau_a(fu_g) = \tau_{\mu}(fa_g)$ for all $f\in C(X)$ and all $g\in G$.
\item Given functions $a,b\colon G\to L^\infty(X,\mu)$ satisfying (i)-(iv) above, show that $\tau_a=\tau_b$ if and only if $a=b$.\ee
\end{ex}

It is worth pointing out which functions $a\colon G\to L^\infty(X,\mu)$ give rise to the canonical extension
$\tau_{\mu}$.

\begin{rem}\label{rem:CanExt}
 Adopt the notation and assumptions of \autoref{thm:Ursu}. Then the function $a\colon G\to L^\infty(X,\mu)$ with $a_1=1$ and $a_g=0$ for all $g\in G\setminus\{1\}$ satisfies conditions (i) through (iv). The associated trace is $\tau_{\mu}$, namely the canonical trace on $C(X)\rtimes G$ induced by $\mu$.
\end{rem}

As mentioned before, in some cases all traces in a crossed product are induced canonically from an invariant probability measure.
The following is the relevant dynamical notion. 
Recall that we denote by $X_g$
the set of fixed points by the group element $g$, namely
$X_g=\{x\in X\colon g\cdot x=x\}$.

\begin{df}\label{df:EssFree}
We say that a dynamical system $G\curvearrowright X$
is \emph{essentially free} if for every $g\in G\setminus\{1\}$ and for every $\mu\in \mathrm{Prob}_G(X)$
we have $\mu(X_g)=0$.
\end{df}

\begin{lma}\label{lma:PropOuter}
Let $(X,\mu)$ be a probability measure space and let
$\alpha\colon X\to X$ be a measure-preserving bimeasurable bijection such that $\mu(\{x\in X\colon \alpha(x)=x\})=0$. Given $\eta \in L^\infty(X,\mu)$, if
$\eta\xi=\eta\alpha(\xi)$ for all $\xi\in L^\infty(X,\mu)$, then $\eta=0$.
\end{lma}
\begin{proof}
Suppose that $\eta\neq 0$ and let $z\in X$ be a point in its (essential) support. The condition on $\alpha$ implies that there
exists a measurable subset $E\subseteq X$ containing $z$ such that $\alpha(E)\cap E=\emptyset$. We can assume that $E\cup \alpha(E)\subseteq \{x\in X\colon \eta(x)\neq 0\}$. Set $\xi=\mathbbm{1}_E$. Then $\eta$ is nonzero on the supports of $\xi$ and $\alpha(\xi)$, so that in the equality
$\eta\xi=\eta\alpha(\xi)$ we can cancel out $\eta$ and deduce that $\xi=\alpha(\xi)$, which is a contradiction. We conclude that $\eta=0$, as desired.
\end{proof}

We are now ready to give a dynamical characterization of when $T(C(X)\rtimes G)$ is canonically affinely homeomorphic to $\mathrm{Prob}_G(X)$. We obtain this as a corollary of \autoref{thm:Ursu}, although it originally appeared much earlier as \cite[Proposition~2.7]{KawTakTom_state_1990}.

\begin{prop} \label{prop:TracesEssFree} 
Let $G\curvearrowright X$ be an amenable dynamical
system. Then every trace on $C(X)\rtimes G$ has the form 
$\tau_{\mu}$ for some $\mu\in \mathrm{Prob}_G(X)$ if and only if 
$G\curvearrowright X$ is essentially free.
\end{prop}
\begin{proof} 
We only prove the ``if'' implication.
Assume that $G\curvearrowright X$ is essentially free.
Let $\tau\in T(C(X)\rtimes G)$, and let $\mu$ be the disintegration of $\tau|_{C(X)}$. Let $a\colon G\to L^\infty(X,\mu)$ be the associated function as 
in \autoref{thm:Ursu}. Given $g\in G\setminus\{1\}$, we have $a_g\xi=a_g\alpha_g(\xi)$, and by \autoref{lma:PropOuter} we deduce that $a_g=0$.
It now follows
from \autoref{rem:CanExt} that $\tau=\tau_{\mu}$. 
\end{proof}

Again in the context of traces on crossed products, 
under the stronger assumption of freeness, it is easier to show
that traces on $C(X)\rtimes_{\lambda} G$ are crorrespond to invariant probability measures on $X$.
We leave the proof of this case as a guided exercise.

\begin{ex}\label{ex:Free2}
Let $G\curvearrowright X$ be a free action. Let $\tau$ be a trace on $C(X)\rtimes_{\lambda} G$. The goal of this exercise is to show that there exists a $G$-invariant probability measure $\mu$ on $X$ such that $\tau=\tau_{\mu}$. 
\be\item Prove that it suffices to check that $\tau=\tau\circ E$. 
\item Argue that it suffices to check this identity on $c_c(G,C(X),
\alpha)$, and even more, on elements of the form $a\mathbbm{1}_g$ with $a\in C(X)$ and $g\in G$.
\item Let $a\in C(X)$ and let $g\in G\setminus\{e\}$. Show that 
$\tau(a\mathbbm{1}_g)=0$. \emph{(With $F=\{g\}$, let $s_1,\ldots,s_n$ be as in part (1) of \autoref{ex:Free}. Check that $E(a\mathbbm{1}_g)=\sum_{k=1}^n s_ka\mathbbm{1}_gs_k^*$ and use this to show that $E(a\mathbbm{1}_g)=0$.)}
\item Deduce that $\tau=\tau\circ E$.
\ee
\end{ex}

\chapter{Amenable groups and the small boundary property} 

In this chapter, we will give an overview of the state-of-the-art results in the setting of amenable groups acting on topological spaces, in relation to $\mathcal{Z}$-stability of crossed products. 
This is by far the least self-contained of the chapters in these 
notes, as we will summarize in a few pages a relatively large body of work in the last decade. Some proofs will be skipped, and some 
will only be sketched in the case of $G=\Z$ acting on a Cantor
space, which often showcases the essential tools without dealing with some of the most technical challenges (stemming both from a 
potentially complicated F{\o}lner structure on $G$, or a complicated topological structure in $X$). We will, however, present the highest generality in the definitions and statements of the results.

\section{Dynamical subequivalence}
In this section, we define and study a way to compare open subsets
in a dynamical system, called \emph{dynamical subequivalence}.
This notion is at the very core of the techniques used in this chapter.

\begin{df}\label{df:DynSubeq}
Let $G\curvearrowright X$ be a dynamical system and let 
$V\subseteq X$ be open and let $K\subseteq X$ be closed. We say that $K$ is \emph{dynamically subequivalent} to $V$, written $K\precsim V$, if there exist a finite 
open cover $K\subseteq U_1\cup \cdots \cup U_n$ and group elements $g_1,\ldots,g_n\in G$ such that $\bigsqcup_{j=1}^n g_jU_j\subseteq V$.

For open sets $U,V\subseteq X$, we say that $U$ is \emph{dynamically subequivalent} to $V$, also written $U\precsim V$, if $K\precsim V$ for all compact subsets $K\subseteq U$.
\end{df}

When working with totally disconnected spaces, comparison will be most relevant for clopen sets, in which case it 
becomes simpler:

\begin{ex}\label{ex:CompClopen}
Let $K\subseteq X$ be clopen and let $V\subseteq X$ be open. Show that
$K\precsim V$ if and only if there are a clopen partition $K=\bigsqcup_{j=1}^n K_j$ and group elements $g_1,\ldots,g_n\in G$ with $\bigsqcup_{j=1}^n g_jK_j\subseteq V$. \end{ex}

There is a very useful connection between dynamical
subequivalence, in the sense above, and Cuntz subequivalence, in the sense of \autoref{df:CtzComp}. We write $\supp_o(a)$ for the open support of a function $a\in C(X)$.

\begin{prop}\label{prop:DynSubeqCtz}
Let $G\curvearrowright X$ be a dynamical system. 
Let $a,b \in C(X)$ with $\supp_o(a)\precsim \supp_o(b)$. Then $a\precsim b$ in $C(X)\rtimes G$.
\end{prop}
\begin{proof}
For the sake of the argument, we will assume that there exist clopen set $U,V\subseteq X$ with $a=\mathbbm{1}_U$ and $b=\mathbbm{1}_V$. 
Find a clopen partition $U=\bigsqcup_{j=1}^n U_j$ and group elements $g_1,\ldots,g_n\in G$ with $\bigsqcup_{j=1}^n g_jU_j\subseteq V$.
Define
\[v =\sum_{j=1}^n u_{g_j}\mathbbm{1}_{U_j} \in C(X)\rtimes G.\]
Using at the third step that $\alpha_{g_j}(\mathbbm{1}_{U_j})\alpha_{g_k}(\mathbbm{1}_{U_k})=0$ whenever $j\neq k$, and using at the fourth step that the support of $\alpha_{g_j}(\mathbbm{1}_{U_j})$ is contained in $V$, we now compute:
\begin{align*}
 v^*bv &=  \sum_{j,k=1}^n \mathbbm{1}_{U_j}u_{g_j}^* \mathbbm{1}_{V} u_{g_k}\mathbbm{1}_{U_k}\\
 &= \sum_{j,k=1}^n u_{g_j}^*\alpha_{g_j}\big(\mathbbm{1}_{U_j}\big)\mathbbm{1}_{V} \alpha_{g_k}\big(\mathbbm{1}_{U_k}\big)u_{g_k}\\
 &= \sum_{j=1}^n u_{g_j}^*\alpha_{g_j}\big(\mathbbm{1}_{U_j}\big)\mathbbm{1}_{V} \alpha_{g_j}\big(\mathbbm{1}_{U_j}\big)u_{g_j}\\
 &=\sum_{j=1}^n u_{g_j}^*\alpha_{g_j}\big(\mathbbm{1}_{U_j}\big)u_{g_j}\\  
 &= \sum_{j=1}^n \mathbbm{1}_{U_j}=a.
\end{align*}
This implies that $a\precsim b$, as desired.
\end{proof}

\begin{prop}\label{prop:UleqVmeasures}
Let $G\curvearrowright X$ be a dynamical system and let 
$U,V\subseteq X$ be open sets satisfying $U\precsim V$. Then 
$\mu(U)\leq \mu(V)$ for all $\mu\in \mathrm{Prob}_G(X)$.
\end{prop}
\begin{proof}
Let $\mu\in \mathrm{Prob}_G(X)$ and let $K\subseteq U$ be a compact subset. We will show that $\mu(K)\leq \mu(V)$, which implies the result by regularity of $\mu$. Let $K\subseteq U_1\cup \cdots \cup U_n$ be a finite open cover and let $g_1,\ldots,g_n\in G$ satisfy $\bigsqcup_{j=1}^n g_jU_j\subseteq V$. Using $G$-invariance of $\mu$ at the second step, we get 
\[\mu(K)\leq \sum_{j=1}^n \mu(U_j)= \sum_{j=1}^n \mu(g_jU_j)\leq \mu(V),\]
as desired.
\end{proof}

Dynamical (strict) comparison for a dynamical system is the property that a weak converse to the above proposition holds; 
see \autoref{df:DynComp}.

\section{Free actions: almost finiteness and $\mathcal{Z}$-stability}\label{sec:af}

Almost finiteness is a topological version of hyperfiniteness
in the context of probability-measure-preserving actions, originally introduced by Kerr in \cite{Ker_dimension_2020}, inspired by a similar notion for ample groupoids due to Matui, in order to establish 
$\mathcal{Z}$-stability of the associated crossed product.

We begin by introducing the notion of a castle.

\begin{df}\label{df:castle}
Let $G\curvearrowright X$ be a dynamical system. A \emph{tower} for $G\curvearrowright X$ is a pair $(S,V)$, where $S\subseteq G$ is a finite subset, and $V\subseteq X$ is an open subset, such that the sets $sV$, for $s\in S$ (called the \emph{levels} of the tower), are pairwise disjoint. The set $S$ is called the \emph{shape} of the tower, and we refer to $V$ as the \emph{base} of the tower. 

A \textit{castle} is a finite collection $\mathcal{C}=\{(S_i,V_i)\}_{i\in I}$ of towers such that all levels of all towers are pairwise disjoint. The disjoint union $\bigsqcup_{i\in I}S_i V_i$ is called the \emph{footprint} of the castle.
\end{df}

Note that the following definition implies amenability of $G$, since we demand the existence of $(K,\varepsilon)$-invariant subsets for every $K$ and every $\varepsilon > 0$.

\begin{df}\label{df:AlmostFinite} (Kerr).
Let $G \curvearrowright X$ be a minimal dynamical system.
We say that the action is \emph{almost finite} if for every 
finite subset $K \subseteq G$, every nonempty, open set $B\subseteq X$ and every $\varepsilon > 0$, there exist exists a  castle $\mathcal{C}=\{(S_i,V_i)\}_{i\in I}$ such that
\begin{enumerate}
\item each level $sV_i$, for $s\in S_i$ and $i\in I$, has diameter at most $\ep$,

\item each shape $S_i$ is $(K,\ep)$-invariant,

\item the complement of the footprint of $\mathcal{C}$ is dynamically subequivalent to $B$, that is,
\[
X\setminus \Big(\bigsqcup_{i\in I}S_i\cdot V_i\Big)\precsim B.
\]
\end{enumerate}
\end{df}

Kerr introduced almost finiteness with the purpose of proving the following result, which is \cite[Theorem~12.4]{Ker_dimension_2020}. Kerr included freeness as one of the assumptions in his theorem (and also in the definition of almost finiteness), although this hypothesis is not used in the proof -- since we will need this result in a more general context, we omit this unnecessary condition.

\begin{thm}\label{thm:afZstab} (Kerr).
Let $G \curvearrowright X$ be a minimal dynamical system which is almost finite. Then $C(X) \rtimes G$ is $\mathcal{Z}$-stable (and hence classifiable).
\end{thm}
\begin{proof}
While the proof is extremely involved, we can give here a sketch of a proof in a very basic case: we will assume $G=\Z$, that $X$ is the Cantor space, and that the action $G\curvearrowright X$ is almost finite with castles consisting of single towers\footnote{This last assumption is by far the least complicated one to remove.}.

In order to check the criterion stated in \autoref{thm:ZstTrZst}, let $F\subseteq C(X)\rtimes G$ be a finite subset, let $\ep>0$ and let $f\in C(X)\rtimes G$ be a nonzero positive element. 
By \autoref{prop:TopFreeCtzSubeq}, we may assume that $f\in C(X)$. Without loss of generality, we may assume that there is a finite subset $K\subseteq G$ such that $F\subseteq C(X)\cup K$.
Fix $n\in \N$ with $\tfrac{1}{n}<\ep$.

Find a clopen tower $(S,V)$ with $(K,\ep)$-invariant shape $S$, levels of diameter at most $\ep$, and such that $X\setminus S\cdot V \precsim \supp_o(f)$. 
Set $N=|S|$. Without loss of generality, we can assume that $S=\{0,\ldots,N\}\subseteq \Z$.
Set $S'\subseteq S$ denote the integer interval obtained from $S$ by removing the first and the last $n+1$ elements.
Upon increasing the size of $S$, we may assume that 
$X\setminus S'\cdot V \precsim \supp_o(f)$.

For $s\in S$, we define a function $f_s\in C(X)$ as follows.
If $s\in \{0,1,\ldots,n\}$, we set $f_s=\tfrac{s}{n} 
\mathbbm{1}_{sV}$. 
If $s\in \{N-n,N-n+1,\ldots,N\}$, we set $f_s=\tfrac{N-s}{n} 
\mathbbm{1}_{sV}$. Finally, if $s\in S'$, then we set $f_s=\mathbbm{1}_{sV}$.

\textbf{Maybe insert picture}

Identify $M_N$ with $\B(\ell^2(S))$ and define a map
$\varphi\colon M_N\to C(X)\rtimes \Z$ by 
\[\varphi(e_{s,t})= f_s u_{s-t}
\]
for all $s,t\in S$. One readily checks that $\varphi$ is a completely positive, contractive order zero map. Moreover,
we have
\[1-\varphi(1)\leq \mathbbm{1}_{X\setminus S'\cdot V}.
\]
Since $X\setminus S'\cdot V \precsim \supp_o(f)$, we deduce from \autoref{prop:DynSubeqCtz} that $1-\varphi(1)\precsim f$ in $C(X)\rtimes\Z$. Finally, the fact that the image of $\varphi$ almost commutes with $F$ is a consequence of $(K,\ep)$-invariance of the shape; we omit the details. This shows that $C(X)\rtimes\Z$ is $\mathcal{Z}$-stable.
\end{proof}

\begin{ex} Complete the details of the proof of 
 \autoref{thm:afZstab} in the case $G=
\Z$ acting on the Cantor space (in the case that the castles do not necessarily have a single tower).
\end{ex}

In view of \autoref{thm:afZstab}, and at least for \emph{free}, minimal actions of amenable groups, in order to prove $\mathcal{Z}$-stability it suffices to establish almost finiteness (and in fact, it is probably equivalent; see \autoref{cnj:Amen}). In the next two sections, we will develop some machinery that allows us to verify almost finiteness in many cases of interest.

\section{Dynamical strict comparison}\label{sec:DynComp}

Dynamical (strict) comparison is a purely dynamical property that will allow us to establish almost finiteness in relevant cases. Rouhgly speaking, it is the property asserting that a weak converse to \autoref{prop:UleqVmeasures} holds, where the inequality $\leq$ is replaced by the strict one $<$.
Explicitly:

\begin{df}\label{df:DynComp} (Kerr).
We say that a minimal dynamical system $G\curvearrowright X$ has 
\emph{dynamical (strict) comparison} if for all open set $U,V\subseteq X$, the condition $\mu(U)<\mu(V)$ for all $\mu\in \mathrm{Prob}_G(X)$ implies $U\precsim V$.
\end{df}

Dynamical comparison is a relatively common phenomenon, and in
fact we do not know whether it is automatic for all amenable, minimal
dynamical systems; see \autoref{pbm:ExActionNoComp}. It should be noted that dynamical comparison, in the way that we defined it,
is only suitable for minimal actions.

\begin{ex} Find an example of a non-minimal dynamical system $G\curvearrowright X$ which does not satisfy \autoref{df:DynComp}.
\end{ex}

It would be interesting to extend this notion to general dynamical systems; see \autoref{pbm:DynCompNonMin}.

Next, we present four results which 
allow us to deduce dynamical comparison in fairly general situations, at least when the group is amenable\footnote{There are also some general results for a large class of nonamenable groups; see \autoref{thm:ParTowComp}.}.

\begin{thm}\label{thm:PetrComp}
Let $G\curvearrowright X$ be a minimal dynamical system. In the following cases, $G\curvearrowright X$ has dynamical comparison.
\be\item (Kerr, \cite[Theorem~9.2]{Ker_dimension_2020}) If $G\curvearrowright X$ is almost finite.
\item (Naryshkin, \cite[Theorem~A]{Nar_polynomial_2022}) If $G$ has polynomial growth.
\item (Downarowicz-Zhang \cite[Theorem~6.33]{DowZha_symbolic_2023}, Kerr-Szabo \cite[Theorem~7.6]{KerSza_almost_2020}) If $G$ has subexponential growth and $\dim(X)<\infty$.
\item (Kerr-Naryshkin, \cite[Theorem~A]{KerNar_elementary_2021}) If $G$ is elementary amenable, $\dim(X)<
\infty$ and $G\curvearrowright X$ is free.
\item (Naryshkin, \cite[Theorem~A]{Nar_group_2024}) If $G$ has an infinite normal subgroup $H$ such that $H\curvearrowright X$ has dynamical comparison, and $G\curvearrowright X$ is free.
\ee
\end{thm}

Freeness is, unfortunately, a crucial assumption to carry out the proof of parts (4) and (5) above.
Dynamical comparison is also known for (amenable) actions of 
certain nonamenable groups; see \autoref{thm:ParTowComp}.

The proofs of all the cases in \autoref{thm:PetrComp} are 
very involved, but we can give a proof of a particular case. We thank Paolo 
Boldrini for his help in writing this argument. 

\begin{thm} Let $X$ denote the Cantor space and let $\Z\curvearrowright X$ be any action. Then $\Z\curvearrowright X$ has dynamical comparison.
\end{thm}
\begin{proof}
Let $\alpha\colon X\to X$ denote the homeomorphism that generates the given action $\Z\curvearrowright X$.
Let $U,V\subseteq X$ be clopen subsets with $\mu(U)<\mu(V)$ for every $G$-invariant measure $\mu$. 
Set $W=U \cap V$ and 
\[\widetilde{U}=U\setminus W \ \ \mbox{ and } \ \ \widetilde{V}=V\setminus W.\] Note that $\mu(\widetilde{U})<\mu(\widetilde{V})$ for every $G$-invariant measure $\mu$. If we show that $\widetilde{U}\precsim \widetilde{V}$, 
then it will immediately follow that $U=\widetilde{U}\sqcup W \precsim \widetilde{V} \sqcup W=V$. 
We can therefore assume that $U$ and $V$ are disjoint.

Given $x\in U$ and $m\in \N$, set 
\[\mu^x_m =\frac{\delta_x+\delta_{\alpha(x)}+\cdots+\delta_{\alpha^m(x)}}{m+1} \ \in \mathrm{Prob}(X).\]

Since the family of intervals $\{1,\dots, m\}$, for $m\in\N$, form a F{\o}lner sequence for $\Z$, it follows that any weak-$\ast$ limit point of the sequence $(\mu^x_m)_{m\in\N}$ is $\alpha$-invariant. In particular, for large enough $m$ we will have $\mu^x_m(U)<\mu^x_m(V)$. Since $\mu_m^x(W)$ is the number of points in the $\{0,\ldots,m\}$-orbit of $x$ which belong to $W$ divided by $m+1$, it follows from the above that the difference
\[S(x,m)=\underbrace{\big|\{k=1,\ldots, m \colon \alpha^k(x)\in V\big|}_{=(m+1)\mu_m^x(V)}
-\big|\{k=1,\ldots ,m \colon \alpha^k(x)\in U\}\big|\] 
is eventually positive. Denote by $\mathsf{m}(x)$ the smallest positive integer for which $S(x,\mathsf{m}(x))>0$ (or, equivalently, for which $S(x,\mathsf{m}(x))=1$). 

\begin{table}[h!]
	\centering
	\begin{tabular}{|c|c|c|c|c|l|l|c|l|c|}
		\hline
		$\operatorname{orb}(x)$ & $x$ & $\alpha(x)$ & $\alpha^2(x)$ & $\alpha^3(x)$ &   $\alpha^4(x)$&$\alpha^5(x)$&$\cdots$ &  $\alpha^{l-1}(x)$&$\alpha^l(x)$\\
		\hline
		$U/V$ & $U$ & $U$ & & $V$&   &$U$&$\cdots$ &  $V$&$V$\\
		\hline
		$S(x,k)$& $0$ & $-1$ & & $0$&   &-1&$\cdots$ &  $0$&$1$ \\
		\hline
	\end{tabular}
	\vspace{.3cm}
	\caption{The first row shows the orbit of $x$ under $\alpha$. The second row records whether each point lies in $U$ or $V$. The third row gives the corresponding values of $S(x,k)$. In this example the first time $S(x,k)=1$ occurs at $k=l$, so $\mathsf{m}(x)=l$.}
	\label{tab:orbit}
\end{table}
\vspace{-.6cm}

The assignment described above defines a function $\mathsf{m}\colon U \to \N$, which is easily seen to be continuous. Since $U$ is compact, there exists $n\in\N$ such that $\mathsf{m}(U)\subseteq \{1,\ldots,n\}$. For $j=1,\ldots,n$, set $U_j=\mathsf{m}^{-1}(j)$, which is a clopen subset of $U$. Then $U=\bigsqcup_{j=1}^n U_j$ and $\alpha^j(U_j)\subseteq V$ by construction.

We claim that the sets $\alpha^j(U_j)$ and $\alpha^k(U_k)$ are disjoint for $j\neq k$; once we do this, we will have shown that $U\precsim V$. Arguing by contradiction, let $j<k$ and assume that there exists 
$y\in \alpha^j(U_j)\cap \alpha^k(U_k)$. Find $x_j\in U_j$ and $x_k\in U_k$ with
$y=\alpha^j(x_j)=\alpha^k(x_k)$, so that $x_j=\alpha^{k-j}(x_k)$. Since $j<k$, we have $S(x_k,k-j-1)\leq 0$; and since $x_j$ belongs to $U_j$, it follows that $S(x_k,k-j)\leq -1$. On the other hand, by definition we also have
\[
1=S(x_j,j)=S(\alpha^{k-j}(x_k),j)=S(x_k,k)-S(x_k,k-j)\geq 1+1=2.
\]
This contradicts the existence of $y$. Therefore the sets $\alpha^i(U_i)$ and $\alpha^j(U_j)$ are disjoint, and $U$ is dynamically subequivalent to $V$.
\end{proof}

\section{The small boundary property}

Just like dynamical comparison, the small boundary property is a purely dynamical condition that is relevant in the study of crossed products. However, unlike dynamical comparison, it was introduced by mathematicians who were not interested in (or probably even knew about) C*-algebras. As such, the fact that the small boundary property is \emph{also} significant to us represents a beautiful connection between dynamical systems and C*-algebras.

Recall that if $G\curvearrowright X$ is a dynamical system, we 
write $\mathrm{Prob}_G(X)$ for the $G$-invariant, Borel probability measures on $X$.

\begin{df}\label{df:SBP}
A dynamical system $G\curvearrowright X$ is said to have the 
\emph{small boundary property} if there exists a basis $\mathcal{B}$ for the topology of $X$ such that 
$\mu(\partial U)=0$ for all $U\in \mathcal{B}$ and all $\mu\in \mathrm{Prob}_G(X)$.
\end{df}

It is obvious that if $X$ is totally disconnected, then any dynamical system $G\curvearrowright X$ has the small boundary property. More generally, the same is true for spaces with finite 
covering dimension, as long as the action is free.

\begin{prop}\label{prop:FinDimSBP} (\cite[Theorem~3.8]{Sza_Zd})
Let $G\curvearrowright X$ be a free dynamical system with $\dim(X)<\infty$. Then $G\curvearrowright X$ has the small boundary property.
\end{prop}

We are now ready to present the main result of this section (\cite[Theorem~6.1]{KerSza_almost_2020}), which gives us access to \autoref{thm:afZstab}.

\begin{thm}\label{thm:afSBP} (Kerr-Szabo)
Let $G\curvearrowright X$ be a free, minimal dynamical system.
Then the following are equivalent:
\be\item $G\curvearrowright X$ is almost finite.
\item $G\curvearrowright X$ has the small boundary property and has dynamical comparison.\ee
\end{thm}

We can now obtain the two most general results that exist
in terms of classifiability of crossed products for free, minimal
actions of amenable groups.

\begin{cor}\label{cor:PolyGrowth}
Let $G$ be a group with polynomial growth and let $G\curvearrowright X$ be a free, minimal dynamical system
with the small boundary property. Then $C(X)\rtimes G$ is $\mathcal{Z}$-stable and thus classifiable.
\end{cor}
\begin{proof}
By \autoref{thm:afSBP} and part~(2) of \autoref{thm:PetrComp}, the dynamical 
system $G\curvearrowright X$ is almost finite. The result thus follows from \autoref{thm:afZstab}.
\end{proof}

If we demand that the covering dimension of $X$ be finite, then the class of groups can be expanded, and the assumption on the small boundary property becomes automatic
by \autoref{prop:FinDimSBP}.

\begin{cor}\label{cor:ElemAmen}
Let $G$ be an elementary amenable group and let $G\curvearrowright X$ be a free, minimal dynamical system
with $\dim(X)<\infty$. 
Then $C(X)\rtimes G$ is $\mathcal{Z}$-stable and thus classifiable.
\end{cor}
\begin{proof}
By \autoref{prop:FinDimSBP}, the the dynamical 
system $G\curvearrowright X$ has the small boundary property. 
Combining \autoref{thm:afSBP} and part~(4) of \autoref{thm:PetrComp}, we deduce that $G\curvearrowright X$ is almost finite. The result thus follows from \autoref{thm:afZstab}.
\end{proof}

It is expected that the above results ought to hold without 
any assumptions
on the groups, except for amenability. Even more, it is expected 
that dynamical comparison should be automatic in our setting, and that $\mathcal{Z}$-stability of the crossed product is equivalent to almost finiteness, and also to the small boundary property:

\begin{cnj}\label{cnj:Amen}
Let $G$ be amenable and let $G\curvearrowright X$ be a free, minimal dynamical system. Then \tfae:
\be\item $C(X)\rtimes G$ is $\mathcal{Z}$-stable;
\item $G\curvearrowright X$ has the small boundary property;
\item $G\curvearrowright X$ is almost finite.\ee
\end{cnj}

We note that (3) implies (1) by \autoref{thm:afZstab}, and (3) implies (2) by \autoref{thm:afSBP}. In the presence of dynamical comparison, (2) and (3) are equivalent, also by \autoref{thm:afSBP}. 

There is a conspicuous lack of tools for obtaining dynamical information about $G\curvearrowright X$
from $\mathcal{Z}$-stability of the crossed product, which means
that it is extremely difficult to deduce either (2) or (3) from (1). 

\section{Beyond freeness: essentially free actions}
The material presented in \autoref{sec:af}, related to almost finiteness, assumes freeness of the action. On the other hand, we saw in \autoref{sec:Simple} that freeness is stronger than necessary in order to obtain simplicity of the crossed product, and that the right assumption under which one ought to study classifiability of crossed products is 
topological freeness. One should thus strive to extend the results of previous sections to the more general context of
topologically free systems.

Although stated for free actions in \autoref{df:AlmostFinite}, the conditions in the definition of almost finiteness do not imply freeness of the action and, in fact, the property holds for many non-free actions as well. 

\begin{ex}
Let $G=\Z\rtimes\Z_2$ act on $S^1$ by letting the copy of $\Z$ 
act via an irrational rotation, and the copy of $\Z_2$ acting 
by complex conjugation. Show that $G\curvearrowright S^1$ is almost finite in the sense of \autoref{df:AlmostFinite}, although
it is not free.
\end{ex}

Almost finiteness does indeed imply some form of freeness, 
as we show next.

\begin{prop}\label{prop:afEssfree}
Let $G\curvearrowright X$ be a minimal, almost finite dynamical 
system. Then $G\curvearrowright X$ is essentially free (see \autoref{df:EssFree}).
\end{prop}
\begin{proof}
Let $g\in G\setminus\{e\}$ and $\mu\in \mathrm{Prob}_G(X)$ be given. 
Fix $\varepsilon>0$ and
set $\mathrm{Fix}(g)=\{x\in X\colon g\cdot x=x\}$. We will show that $\mu(\mathrm{Fix}(g))<\varepsilon$.

Let $\varepsilon'>0$ satisfy $1-(1-\varepsilon')^2=\varepsilon$. Find an open castle $\{(V_i,S_i)\}_{i\in I}$ with shapes that are $(\{g^{-1}\},\varepsilon')$-invariant, and satisfying, in particular, 
\begin{equation}\label{2.4}\tag{2.4}
\mu\Big(\bigsqcup_{i\in I}S_i\cdot V_i\Big)\ge 1-\varepsilon'.
\end{equation}
Given $i\in I$, it is easy to see that $(S_i\cap g^{-1}S_i)\cdot V_i$ is disjoint from $\mathrm{Fix}(g)$. Thus $\bigsqcup_{i\in I}(S_i\cap g^{-1}S_i)\cdot V_i\subseteq X\setminus\mathrm{Fix}(g)$, and we have
\begin{align*}
1-\mu(\mathrm{Fix}(g))&\ge\sum_{i\in I}\big| \, S_i\cap g^{-1}S_i\big|\mu(V_i)\\
&\ge(1-\varepsilon')\sum_{i\in I}|S_i| \, \mu(V_i)\\
&=(1-\varepsilon')\mu\Big(\bigsqcup_{i\in I}S_i\cdot V_i\Big)\\
&\stackrel{\text{\eqref{2.4}}}{\ge}{(1-\varepsilon')^2}=1-\varepsilon.
\end{align*}
As $\varepsilon$ is arbitrary, this shows that $\mu(\mathrm{Fix}(g))=0$. Thus $G\curvearrowright X$ is essentially free.
\end{proof}

It follows that essential freeness, which we introduced in \autoref{sec:Traces} in order to understand traces on crossed products, also arises naturally in the context of the available tools, in particular almost finiteness. This suggests that the existing machinery could be extended to the essentially free setting, a task that was initiated in \cite{GGGKN_essential_2024}.

Extending results from the free to the essentially free setting one may not seem to be a significant advance at first sight, but there are in general a number of technical difficulties when dealing with non-free actions. One of the most immediate problems that one encounters is that the restriction of an essentially free action to a subgroup is not necessarily essentially free (unlike the situation for free actions), thus obstructing a number of arguments involving extensions of groups and restrictions of actions.

\begin{ex}
Find an example of an essentially free dynamical system $G\curvearrowright X$ and a subgroup $H\leq G$
such that $H\curvearrowright X$ is not essentially free.
\end{ex}

The following extends \autoref{prop:FinDimSBP} from free actions to essentially free ones. 

\begin{prop}
Let $G\curvearrowright X$ be an essentially free dynamical system with $\dim(X)<\infty$. Then $G\curvearrowright X$ has the small boundary property. 
\end{prop}

\begin{thm} \label{thm:EssFreeSBP}
(Gardella-Geffen-Gesing-Kopsacheilis-Naryshkin).
Let $G\curvearrowright X$ be an essentially free, 
minimal dynamical system. Then the following are equivalent:
\be\item $G\curvearrowright X$ is almost finite.
\item $G\curvearrowright X$ has the small boundary property and has dynamical comparison.\ee
\end{thm}

The following is a version of \autoref{cor:PolyGrowth} for 
essentially free actions.

\begin{cor}
Let $G$ be a group with polynomial growth and let $G\curvearrowright X$ be an essentially free, minimal dynamical system
with the small boundary property. Then $C(X)\rtimes G$ is $\mathcal{Z}$-stable and thus classifiable.
\end{cor}
\begin{proof}
By \autoref{thm:EssFreeSBP} and \autoref{thm:PetrComp}, the dynamical 
system $G\curvearrowright X$ is almost finite. The result thus follows from \autoref{thm:afZstab}.
\end{proof}

We do not obtain an analog of \autoref{cor:ElemAmen} for essentially free actions,
because the proof of dynamical comparison in that setting (part~(4) of \autoref{thm:PetrComp}) requires freeness in a crucial way\footnote{Among others, it relies on a bootstrap argument that requires one to know that the restriction of an action as in the statement to any subgroup is again an action as in the statement -- this is true for freeness, but not for essential freeness.}.

\section{Beyond essential freeness: allosteric actions}

In the previous section, we saw that some of the strongest
results around almost finiteness, originally proved for free actions, can be extended to the essentially free setting.
In our quest to develop machinery to deal with 
\emph{topologically} free actions, a natural question arises: 
do there exist minimal, topologically free actions that are not essentially free? Following terminology introduced by Joseph \cite{Jos_amenable_2024}, we give such actions a name.

\begin{df} A dynamical system $G\curvearrowright X$ is said to be
\emph{allosteric} if it is minimal, topologically free, but not essentially free. A group $G$ is said \emph{allosteric} if it admits an allosteric action. 
\end{df}

Many of the most natural amenable groups are not allosteric. For example, the 
following is \cite[Corollary~2.4]{Jos_continuum_2024}

\begin{prop} 
If a discrete group $G$ has only countably many subgroups, then $G$ is not allosteric.
\end{prop}

The class of groups with only countably many subgroups is significant: it includes all finitely generated groups with polynomial growth, and more generally polycyclic
groups, extensions of Noetherian groups by groups with only countably many
subgroups, etc. 

\begin{ex}
Show that a finitely generated group with polynomial growth has at most countably many subgroups.
\end{ex}

Thus, allosteric actions of amenable groups are particularly difficult to come by, but they were recently shown to exist
by Joseph in \cite{Jos_amenable_2024}. These examples show that the techniques currently available to study $\mathcal{Z}$-stability of crossed products are not sufficient.

Joseph's allosteric actions act on the Cantor space, and thus have the small boundary property. They also have dynamical strict comparison (since they are profinite actions), and thus it is natural to ask whether their crossed products are $\mathcal{Z}$-stable (although they are not almost finite). This is indeed the case; see \cite[Theorem~5.8]{GGGKN_essential_2024}.

\begin{thm}\label{thm:Allosteric} (Gardella-Geffen-Gesing-Kopsacheilis-Naryshkin).
Let $G\curvearrowright X$ be one of the allosteric actions constructed by Joseph in \cite[ Theorem 3.2 and Theorem 4.1]{Jos_amenable_2024}. Then $C(X)\rtimes G$ is $\mathcal{Z}$-stable.
\end{thm}

It is desirable to construct more examples of allosteric actions of amenable groups in order to be able to develop techniques that allow us to show $\mathcal{Z}$-stability of their crossed products; see \autoref{qst:AFallosteric}.

\chapter{Paradoxical groups and purely infinite C*-algebras}

\section{Purely infinite crossed products}

In this section, we collect a number of elementary definitions and results
that will be needed in the rest of the chapter. We begin by defining pure 
infiniteness for simple C*-algebras, and refer the reader to the second
section in the appendix for more background on this class, as well as for 
some context within Elliott's program.

\begin{df}\label{df:PI}
A simple, unital C*-algebra $A$ is said to be \emph{purely infinite} if 
for every element $a\in A$,
there exist $x,y\in A$ with $xay=1$.
\end{df}

Pure infiniteness, in the simple case, can be characterized in
terms of Cuntz subequivalence in a neat way.

\begin{ex}\label{ex:CharPICtzSub}
Let $A$ be a simple, unital C*-algebra. Show that
$A$ is purely infinite if and only if $1\precsim a$ for every
positive element $a\in A$.
\end{ex}

The relevance of pure infinite C*-algebras in our setting stems from the fact
that if $G$ is nonamenable and $G\curvearrowright X$ is amenable, minimal and 
topologically free, then the crossed product $C(X)\rtimes G$ is traceless by
\autoref{prop:GamenIFFinvmeas}. In this setting, Kirchberg's dichotomy asserts that $\mathcal{Z}$-stability for
$C(X)\rtimes_{\alpha} G$ is equivalent to pure infiniteness, which in practice can be
checked more directly than $\mathcal{Z}$-stability. This is perhaps a bit surprising at first sight, since pure infiniteness, in the simple and nuclear setting, is a much stronger condition than $\mathcal{Z}$-stability, and hence it may seem paradoxical that it is easier to check the stronger condition in practice. The reason is that it is far from obvious that pure infiniteness implies $\mathcal{Z}$-stability, and this in fact follows from an extremely deep result of Kirchberg (\cite{KirPhi_embedding_2000}): a simple, purely infinite, nuclear C*-algebra automatically absorbs the Cuntz algebra $\mathcal{O}_\infty$ tensorially. 
We will tacitly use this fact throughout.

The next proposition implies that in order to establish pure infiniteness of a traceless crossed product (such as the crossed product of an amenable action of a nonamenable group; see \autoref{prop:GamenIFFinvmeas}),
it suffices to show that the dynamical system has comparison.

\begin{prop}\label{prop:CPnoTracesPI}
Let $G\curvearrowright X$ be an amenable, minimal, topologically free dynamical system. If $G$ is nonamenable and $G\curvearrowright X$ has comparison,
then $C(X)\rtimes G$ is purely infinite and simple.
\end{prop}
\begin{proof}
By \autoref{thm:SimpleCP}, we only need to show pure infiniteness. 
Let $a\in C(X)\rtimes  G$ be a nonzero positive contraction; we will show that 
$1\precsim a$ in $C(X)\rtimes  G$. Once we do this, the result will follow from \autoref{ex:CharPICtzSub}.
By \autoref{prop:TopFreeCtzSubeq},
there exists a nonzero positive contraction $f\in C(X)$ such that 
$f\precsim a$ in $C(X)\rtimes G$. 

Set $W=\{x\in X\colon f(x)>0\}$, and
choose $U$ to be a nonempty open subset of $W$ such that $\overline{U}\subseteq W$.
There exists a positive contraction $g\in C(X)$ such that $g=0$ on $X\setminus W$ and $g=1$ on $\overline{U}$.
Note that $X\precsim U$, since $G\curvearrowright X$ has 
comparison and there are no $G$-invariant measures on $X$ 
by \autoref{prop:CPnoTracesPI}.
Now \autoref{prop:DynSubeqCtz} implies that
$1_{C(X)}\precsim g$ in $C(X)\rtimes G$.
On the other hand, we have 
$g\precsim f$ since $\{x\in X\colon g(x)>0\}\subseteq W=\{x\in X\colon f(x)>0\}$; see \autoref{ex:CtzCompCX}. Transitivity of Cuntz subequivalence (see \autoref{ex:CtzCompTransitive}) gives $1_{C(X)}\precsim f$, and so $1_{C(X)}\precsim a$, as desired.
\end{proof}

In short, if $G$ is nonamenable and $G\curvearrowright X$ is an amenable, minimal and topologically free dynamical system, then in order to establish $\mathcal{Z}$-stability of $C(X)\rtimes G$, it suffices to prove comparison for $G\curvearrowright X$. This will be our task for the rest of the chapter.

\section{Groups with paradoxical towers}

In this section, we introduce a class of nonamenable groups for which we will show 
in \autoref{thm:ParTowComp} that every minimal, amenable action has dynamical comparison.
The following definition was introduced in \cite{GGKN_classifiability_2023}.

\begin{df}\label{df:ParTow} (Gardella-Geffen-Kranz-Naryshkin). Let $G$ be a discrete group 
and let $n\in\N$. We say that $G$ has \emph{$n$-paradoxical towers} if for every finite subset $D\subseteq G$ there
exist subsets $A_1,\ldots, A_n\subseteq G$ and group elements 
$g_1,\ldots, g_n\in G$ such that 
\be\item $G=\bigcup_{j=1}^n g_jA_j$, and 
\item the sets $dA_j$, for $d\in D$ and $j=1,\ldots,n$, are pairwise disjoint.\ee

We say that $G$ has \emph{paradoxical towers} if it has $n$-paradoxical towers for some $n\in\N$.
\end{df}

The following lemma justifies the terminology.

\begin{lma} \label{lma:ParTowPar}
Groups with paradoxical towers are paradoxical.\end{lma}
\begin{proof} Let $G$ be a group with $n$-paradoxical towers. Arguing by contradiction, assume that there exists a left-invariant, finitely-additive probability measure $\mu$ defined on the power set of $G$.
Let $D$ be any finite set with $|D|>2$, and 
find $A_1,\ldots, A_n\subseteq G$ and $g_1,\ldots, g_n\in G$ such that 
\be\item[(5.1)] $G=\bigcup_{j=1}^n g_jA_j$, and 
\item[(5.2)] the sets $dA_j$, for $d\in D$ and $j=1,\ldots,n$, are pairwise disjoint.\ee 
Then
\[1=\mu(G)\stackrel{(5.1)}{\leq} \sum_{j=1}^n \mu(A_j)
 \stackrel{(5.2)}{\leq}\frac{1}{|D|} <1,
\]
which is a contradiction.
\end{proof}

As an elementary example, we show that free groups have 2-paradoxical towers.

\begin{prop}\label{prop:F2ParTow}
The free group $\mathbb{F}_n$ has 2-paradoxical towers.
\end{prop}
\begin{proof}
For simplicity, we will assume that $n=2$, although the proof is completely general.
We will show that given a finite subset $D\subseteq \mathbb{F}_2$ there are subsets
$A_1,A_2,A_3\subseteq \mathbb{F}_2$ and $g_1,g_2,g_3\in \mathbb{F}_2$
such that: 
\be\item[(i)] the sets $dA_j$, for $d\in D$ and $j=1,2,3$, are pairwise disjoint,
\item[(ii)] the sets $\mathbb{F}_2\setminus g_jA_j$, for $j=1,2,3$, are pairwise
disjoint.
\ee

Before doing this, we show that the above implies that 
$\mathbb{F}_2$ has 2-paradoxical towers. Indeed, condition (ii)
gives
\[\emptyset=\big(\mathbb{F}_2\setminus g_1A_1\big)\cap \big(\mathbb{F}_2\setminus g_2A_2\big),\]
and by taking complements we get $\mathbb{F}_2=g_1A_1\cup g_2A_2$. In particular, we deduce that
$A_1,A_2$ and $g_1,g_2$ satisfy the conditions in \autoref{df:ParTow} for $n=2$.

Denote by $a,b$ the generators of $\mathbb{F}_2$, and set 
$L=\{a,b,a^{-1},b^{-1}\}$. For $h\in \mathbb{F}_2$, we write
$\mathsf{W}(h)\subseteq \mathbb{F}_2$ for the set of all reduced words 
with letters from $L$ that start with $h$. If
$x\in L$ is the last letter of $h$, then a generic element of $\mathsf{W}(h)$ has the
form $hg$ with $g\notin \mathsf{W}(x^{-1})$. In particular,
\begin{equation}\label{eqn:ComplementF2}\tag{5.3}
\mathbb{F}_2\setminus h^{-1}\mathsf{W}(h)= \mathsf{W}(x^{-1}).
\end{equation}
For $r\geq 0$, we write $B_r$ for the set 
of all reduced words with at most $r$ letters, and observe that $B_rB_s= B_{r+s}$ for 
all $r,s\geq 0$.
Let $D\subseteq \mathbb{F}_2$ be a finite set.
Find $m\geq 0$ with $D\subseteq B_m$, and define 
\[g_1=a^{2m}ba, \ \ g_2=a^{2m}ba^{-1}, \ \ \mbox{ and }
\ \ g_3=a^{2m}b^2.
\]
For $j=1,2,3$, set $A_j=\mathsf{W}(g_j)$. 

We claim that these sets satisfy condition (i).
Since every element of $D$ has length at most $m$, it suffices to check
that whenever $e\in B_{2m}$ and $j,k=1,2,3$ satisfy 
$eA_j\cap A_{k}\neq \emptyset$, then $e=1$ and $j=k$.
If $eA_j$ intersects $A_k$, then there is 
$g\in A_j$ such that $eg$, after reduction, starts with $2m$ copies of $a$. 
Since the $(2m+1)$-st letter of $g$ is not $a$, it follows that 
the product
of $e$ and $g$ cannot have any cancelations, and thus
$e=a^{k}$ for some $0\leq \ell \leq m$. 
On the other hand, $a^{\ell}g_j$ never has $g_k$ as an initial segment unless $\ell=0$
and $j=k$. This proves the claim.

Using \eqref{eqn:ComplementF2}, we get 
\[\mathbb{F}_2\setminus g_1A_1=\mathsf{W}(a^{-1}), \ \ 
 \mathbb{F}_2\setminus g_2A_2=\mathsf{W}(a), \ \ \mbox{ and } \ \
 \mathbb{F}_2\setminus g_3A_3=\mathsf{W}(b^{-1}), \ \
\]
and these sets are clearly pairwise disjoint. This finishes the proof.
\end{proof}

With a bit more work, one can show the following:

\begin{egs} The following nonamenable groups have
 paradoxical towers:
 \be\item Hyperbolic groups.
 \item SL$_n(\Z)$, for all $n\geq 2$.
 \item Nontrivial free products $G\ast H$.
 \item Baumslag-Solitar groups.
 \item Acylindrically hyperbolic groups such as $\mathrm{Out}(\mathbb{F}_n)$.
 \ee
\end{egs}

What the groups listed above have in common is that, for some $n\in\N$, 
each of them admits a topologically free $n$-filling action on a Hausdorff
Baire space; see \cite[Proposition~4.6]{GGKN_classifiability_2023}.

On the other hand, not all paradoxical groups admit paradoxical towers:

\begin{ex} Show that $\mathbb{F}_2\times \Z$ does not have paradoxical towers. \end{ex}

\section{Dynamical comparison and classifiability}

The following is the main result of this chapter, and it is the key step in showing pure infiniteness of the crossed product.\footnote{One also needs to invoke, in some way or another, Kirchberg's extremely deep result, asserting that a purely infinite, simple, nuclear \ca\ must absorb $\mathcal{O}_\infty$ tensorially, in order to go from pure infiniteness to $\mathcal{Z}$-stability; see the discussion after \autoref{ex:CharPICtzSub}.} A complete proof (of a more general statement) can be found in \cite[Theorem~3.6]{GGKN_classifiability_2023}.

\begin{thm}\label{thm:ParTowComp}
Let $G$ be a group with paradoxical towers and let $G\curvearrowright X$ be an amenable, minimal dynamical
system. Then $G\curvearrowright X$ has dynamical comparison.
\end{thm}
\begin{proof} 
In order to illustrate the argument more clearly, we
will prove this for $G=\mathbb{F}_2$.
Given an open set $U\subseteq X$, we need to show that 
$X\precsim U$; see \autoref{df:DynSubeq}. We will construct an open
set $V\subseteq X$ satisfying $X\precsim V\precsim U$, which is easily seen to imply $X\precsim U$.
By minimality, there is a finite set $D\subseteq \mathbb{F}_2$
such that $D\cdot U=X$. 
Use \autoref{prop:F2ParTow}
to find sets $A_1,A_2\subseteq \mathbb{F}_2$ and group 
elements $g_1,g_2\in \mathbb{F}_2$ satisfying the 
following:
\begin{enumerate}
\item[(i)] The sets $dA_j$, for $d\in D^{-1}$ and $j=1,2$, are 
pairwise disjoint.
\item[(ii)] $(\mathbb{F}_2\setminus g_1A_1)\cap (\mathbb{F}_2\setminus g_2A_2)=\emptyset$.
\end{enumerate}
Let $\varepsilon<\frac{1}{24}$, and use amenability
of $\mathbb{F}_2\curvearrowright X$ to find a continuous map 
$\mu\colon X\to \mathrm{Prob}(\mathbb{F}_2)$ with
\begin{enumerate}
\item[(iii)]$\sup_{x\in X}\|\mu(g\cdot x)-g\cdot \mu(x)\|_1<\varepsilon$
for all $g\in D^{-1}\cup \{g_1, g_2\}$.
\end{enumerate}

For $j=1,2$, set 
\[V_j=\Big\{x\in X\colon \mu(x)(A_j)>\tfrac{2}{3}+\varepsilon\Big\},\]
and define $V=V_1\cup V_2$. Combining (ii) and (iii), it is 
easy to see that the sets $d\cdot V_j$, 
for $d\in D$ and $j=1,2$, are pairwise disjoint.
In particular $V=V_1\sqcup V_2$. 

We will now show that $V\precsim U$. For each $d\in D$, set $V_d=V\cap d^{-1} U$. 
Then the sets $d\cdot V_d$, for $d\in D$, are pairwise disjoint
and contained in $U$. This shows that
$V\precsim U$.

To show that $X\precsim V$, we define open sets
\[W_j=\Big\{x\in X\colon \mu(x)(\mathbb{F}_2\setminus g_jA_j)
<\tfrac{2}{3}-2\varepsilon\Big\},\]
for $j=1,2$. By (ii) and the fact that $\mu(x)$ has total mass 1
for all $x\in X$, it follows that $W_1\cup W_2=X$. Moreover, 
another use of (iii) shows that
$g_j^{-1}W_j\subseteq V_j$ for $j=1,2$.
This shows that $X\precsim V$, and thus $X\precsim U$, as desired.
\end{proof}

We can now establish classifiability of crossed products by groups with 
paradoxical towers.

\begin{cor}
Let $G$ be a group with paradoxical towers and let $G\curvearrowright X$ be an amenable, minimal, topologically free dynamical system. Then $C(X)\rtimes G$ is classifiable.
\end{cor}
\begin{proof}
By \autoref{thm:ParTowComp}, the system $G\curvearrowright X$
has comparison. Since $G$ is nonamenable by \autoref{lma:ParTowPar}, it follows from the equivalence between (2) and (3) in \autoref{prop:GamenIFFinvmeas} that there are 
no $G$-invariant measures on $X$, and hence $C(X)\rtimes G$ is purely infinite by
\autoref{prop:CPnoTracesPI}, and thus classifiable by Kirchberg's work.
\end{proof}

\chapter{Open problems and questions}

In this concluding chapter, we present a selection of open problems related to crossed products and dynamical system that will likely shape research in the coming years. Some of these are highly demanding and should be viewed as long-term projects, while with some others, results in the short term are more realistic. Likewise, some problems are phrased broadly, whereas others are posed in more precise and concrete terms.\footnote{Just a disclaimer: there is no attempt of being exhaustive in listing all possible, interesting questions and problems. I may have omitted some that I find myself interesting.}

\section{More about classifiable crossed products}

Although we have a fairly clear picture of how classifiability ought to look like in the case of (essentially) free actions of amenable groups, we still do not know much about the case of allosteric actions. There are many interesting problems to explore, among which we highlight these:

\begin{pbm} For one of the allosteric actions $G\curvearrowright X$ constructed by Joseph, explicitly compute the trace space of $C(X)\rtimes G$.
\end{pbm}

\begin{pbm} Construct more examples of allosteric actions of amenable groups. 
\end{pbm}

It would be highly desirable to develop machinery that allowed us to treat allosteric actions in a systematic way. We thus ask:

\begin{qst} \label{qst:AFallosteric}
Is there a version of almost finiteness that 
applies to allosteric actions?
Does this version of almost finiteness imply $\mathcal{Z}$-stability in the minimal case? 
\end{qst}

Progress in the nonamenable setting has been somewhat stuck since the results in 
\cite{GGKN_classifiability_2023}, but even within the class for which 
classifiability is known, there are interesting things to explore:

\begin{pbm} Find more examples of groups with paradoxical towers. 
\end{pbm}

All the examples we know of groups with paradoxical towers contain a free group. While it is not difficult to show, using the ping-pong lemma, that a torsion-free group with 2-paradoxical towers always contains $\mathbb{F}_2$, whether this is true in general is not known.

\begin{qst}
Do all groups with paraodical towers have to contain
$\mathbb{F}_2$?
\end{qst}

Again, it would be desirable to develop machinery that allowed us to deal with all amenable actions at once, regardless of whether the acting group is amenable or not.

\begin{qst} Is there a version of almost finiteness that 
does not incorporate amenability of $G$ into the definition but rather amenability of $G\curvearrowright X$? Do amenable dynamical systems (with dynamical comparison) of paradoxical groups satisfy this property?
Does this version of almost finiteness imply $\mathcal{Z}$-stability in the minimal case? 
\end{qst}

\section{Extending the basic theory}

It is also extremely valuable to expand our understanding of crossed products in 
a general setting, not necessarily under our standard assumptions of amenability, minimality, or (topological) freeness.

Our first problem consists in obtaining a generalization of 
\autoref{thm:SimpleCP} that does not assume amenability of the action.

\begin{pbm} \label{qst:Ozawa}
Give a dynamical (and possibly also group-theoretic) characterization of simplicity for the crossed product of a dynamical system.
\end{pbm}

Some work exists \cite{AmrUrs_generalized_2022}, but Ozawa's question remains unanswered: is $G\curvearrowright X$ C*-simple if and only if $G\curvearrowright X$ is minimal and every open stabilizer is C*-simple?

\begin{pbm}\label{pbm:DynCompNonMin}
Define a notion of dynamical strict comparison which is meaningful to non-minimal dynamical systems. Can one obtain generalizations of the results at the end of \autoref{sec:DynComp} for this notion?
\end{pbm}

The following question is related to the above:

\begin{qst}\label{pbm:ExActionNoComp} Is there a minimal, amenable dynamical system which does not have dynamical comparison?
\end{qst}

A positive answer would be extremely
interesting even if the action is not topologically free. 
An example of a minimal action on a free group on a totally disconnected space without comparison was recently constructed by Boldrini and Prasad.

\section{Non $\mathcal{Z}$-stable crossed products}

Moving beyond classifiable crossed products, a natural question is the following:
\emph{How wild can crossed products be when they are not $\mathcal{Z}$-stable?} We suspect that the exotic behavior that the examples of R{\o}rdam \cite{Ror_simple_2003} or Villadsen \cite{Vil_simple_1998} exhibit cannot be realized by crossed products. In other words,
when a crossed product is simple, separable, unital and nuclear, but perhaps not $\mathcal{Z}$-stable, we suspect that a certain amount of regularity still holds. As a concrete example, the following question has been receiving increased attention lately.

\begin{qst} Does $C(X)\rtimes G$ always have stable rank one if $G\curvearrowright X$ has an invariant measure, is minimal and topologically free (but does not necessarily have the small boundary property)? 
\end{qst}

For crossed products which are $\mathcal{Z}$-stable, the question above has a positive answer by work of R{\o}rdam \cite{Ror_stable_2004}. The case where $G=\Z^d$ has been answered affirmatively by Guang-Niu (\cite[Corollary~7.15]{GuaNiu_stable_2020}), who also showed that the answer is also positive whenever $G\curvearrowright X$ has dynamical comparison and the so-called uniform Rokhlin property (which is automatic for minimal actions of $\Z^d$, by results in \cite{GuaNiu_stable_2020}). These results have been extended by Naryshkin to all abelian groups (indeed also for FC-groups); see \cite[Theorem~1.3]{Nar_URPC_2024}. More recently, Bell-Geffen-Kerr have improved these results in to product actions (\cite[Theorem~C]{BelGefKer_stable_2025}), and they also showed that the crossed product of a generic minimal, topologically free action of a free group on the Cantor space has stable rank one (\cite[Corollary~B]{BelGefKer_stable_2025}).

Continuing with non-classifiable simple crossed products, the computation of the radius of comparison may turn out to have a 
dynamical answer (\cite{HirPhi_radius_2022, Phi_minimal_2016}):

\begin{cnj} (Phillips-Toms). Let $G\curvearrowright X$ be a
free, minimal dynamical system of an amenable group $G$. Then 
\[\mathrm{rc}(C(X)\rtimes G)=\tfrac{1}{2} \mathrm{mdim}(G\curvearrowright X).
\]
\end{cnj}

While making further progress in the above conjecture may require new insights, it would be very interesting to explore similar questions for topologically free actions.

\begin{qst} 
Let $G\curvearrowright X$ be a
topologically free, minimal dynamical system of an amenable group $G$.  Does the radius of comparison of the crossed product have any relation to the radius of comparison beyond the free case?
\end{qst}

\appendix

\chapter{Structure and classification of C*-algebras}\label{app:Elliott}

In this appendix, we give an overview of the historical developments surrounding one of the main
driving forces in C*-algebra theory in the last three decades: the Elliott classification
programme. The material contained here will help in the understanding and contextualization of the
topics covered in these notes.
No proofs will be given here, and we will refer the reader to the literature for further details
instead.

\section{The beginnings of the Elliott programme}
In the early 1990s, George Elliott initiated a programme to classify simple, separable, nuclear
C*-algebras in terms of $K$-theory. The evidence available at the time was in retrospect
rather limited, and it consisted of Glimm's classification of UHF-algebras \cite{Gli_certain_1960};
the subsequent
work of Elliott on AF-algebras \cite{Ell_classification_1976}; and Elliott's classification of
A$\mathbb{T}$-algebras\footnote{An \emph{A$\mathbb{T}$-algebra} is a direct limit of algebras of the
form $F\otimes C(\mathbb{T})$, where $F$ is a finite dimensional \ca.}
of real rank zero \cite{Ell_classification_1993}. The precise form of the
invariant saw some changes once it was realized that $K$-theory by itself would not be enough to
classify algebras without sufficiently many projections. This invariant is now known as the \emph{Elliott
invariant}, and is defined as follows.

\begin{df}\label{df:EllInv}
Let $A$ be a unital \ca. We define its \emph{Elliott invariant} $\mathrm{Ell}(A)$ to be the
quadruple
\[\mathrm{Ell}(A)=((K_0(A),K_0(A)_+,[1_A]_0),K_1(A),T(A),\rho),\]
where $\rho\colon K_0(A)\times T(A)\to \mathbb{R}$ is the canonical pairing determined by
$\rho([p]_0,\tau)=\tau(p)$ for all projections $p\in A\otimes \K$ and all traces $\tau\in T(A)$.
\end{df}

There is a natural notion of equivalence of Elliott invariants, where the objects involved are
supposed to be isomorphic in such a way that the pairings of the different tuples fit into a
commutative diagram.

Elliott's original conjecture is the following:

\begin{cnj}\label{cnj:EC}
Let $A$ and $B$ be simple, separable, unital, nuclear \ca s. Then $A$ and $B$ are isomorphic if
and only if $\mathrm{Ell}(A)\cong \mathrm{Ell}(B)$.
\end{cnj}

Some words about the conditions in \autoref{cnj:EC} are in order. Unitality is the least crucial
of all assumptions, and it is in fact not strictly necessary (although the invariant in the non-unital
case is slightly different). Simplicity is a very common assumption throughout mathematics, and the
classification of simple \ca s is a very natural place to start if one wishes to classify more
complicated algebras. Finally, separability and nuclearity are included as assumptions because
one should not expect to obtain reasonable classification results for C*-algebras whose weak
closures in their GNS representations are not classified or even properly understood. Weak closures
in GNS constructions are separable (in the sense of tracial von Neumann algebras) if the original
C*-algebra is separable, and they are all hyperfinite if and only if the C*-algebra is nuclear.
(And if the state is an extreme trace, then the weak closure is a factor, indeed the hyperfinite II$_1$-factor
$\mathcal{R}$.)
In particular, separability and nuclearity are conditions that guarantee the classifiability of the
weak closures.

In this formulation, Elliott's conjecture saw remarkable success in the 1990's. For example,
Elliott's classification of AF-algebras is a confirming example, since for unital AF-algebras the
Elliott invariant can be seen to be equivalent to the triple $(K_0(A),K_0(A)^+,[1_A]_0)$. The
same is true for Elliott's classification of A$\mathbb{T}$-algebras of real rank zero in terms
of $K_0$ and $K_1$.

\section{Kirchberg algebras: a major success, modulo the UCT}

Perhaps the most outstanding success of this programme came with the classification of purely
infinite \ca s, so we define this notion next.

\begin{df}\label{df:PurelyInf}
A simple, unital \ca\ $A$ is said to be \emph{purely infinite} if for all $a\in A\setminus\{0\}$ there exist
$x,y\in A$ such that $xay=1$.
\end{df}

Examples of purely infinite simple algebras are the Cuntz algebras $\mathcal{O}_n$ (these are also nuclear
and separable), as well as the Calkin algebra (which is not nuclear, or even separable). Purely
infinite simple \ca s do not admit any trace, so their Elliott invariant reduces to $K$-theory.

A highly praised result of Kirchberg and Phillips \cite{Kir_classification_?, Phi_classification_2000}
from the late 1990's confirms Elliott's conjecture in the purely infinite case, under the additional assumption of the Universal Coefficient Theorem (UCT). This result builds
on vast amounts of work of several authors, among which Kirchberg stands out. For this reason, a
purely infinite, simple, separable and nuclear \ca s is commonly referred to as a \emph{Kirchberg}
algebra.

\begin{thm}\label{thm:PurInfClass}(Kirchberg-Phillips).
Let $A$ and $B$ be unital Kirchberg algebras satisfying the UCT.
Then $A$ and $B$ are isomorphic if and only if $(K_0(A), [1_A]_0)\cong (K_0(B), [1_B]_0)$ and $K_1(A)\cong K_1(B)$.
\end{thm}

The Cuntz algebras $\Ot$ and $\OI$ played a pivotal role in the proof of the above theorem (at least
in the approach taken in \cite{Phi_classification_2000}), thanks to the following facts, known as
Kirchberg's Geneva Theorems \cite{KirPhi_embedding_2000}:
\bi\item If $A$ is a Kirchberg algebra, then $A\otimes\OI\cong A$;
\item If $A$ is a simple, separable, unital, exact \ca, then $A\otimes \Ot\cong \Ot$.\ei

The reader will note that \autoref{thm:PurInfClass} has an unexpected assumption: the UCT. Accordingly, we
define it next.

\begin{df}\label{df:UCT}
Let $A$ and $B$ be separable \ca s. We say that the pair $(A,B)$ \emph{satisfies the UCT} if the following conditions are satisfied:
\be\item The natural map $\tau_{A,B}\colon KK(A,B)\to \Hom(K_\ast(A),K_\ast(B))$ defined by Kasparov
in \cite{Kas_equivariant_1988}, is surjective.
\item The natural map $\mu_{A,B}\colon \ker(\tau_{A,B})\to \Ext(K_\ast(A),K_{\ast+1}(B))$ is an isomorphism.\ee
If this is the case, by setting $\ep_{A,B}=\mu_{A,B}^{-1}\colon \Ext(K_\ast(A),K_{\ast+1}(B)) \to KK(A,B)$, we obtain a short exact sequence
$$\xymatrix@=1.35em{0\ar[r] &
\mathrm{Ext}(K_\ast(A),K_{\ast+1}(B))\ar[r]^-{\varepsilon_{A,B}}&
KK(A,B)\ar[r]^-{\tau_{A,B}}&\mathrm{Hom}(K_\ast(A),K_\ast(B))\ar[r]&0,}$$
which is natural on both variables because so are $\tau_{A,B}$ and $\mu_{A,B}$.

We further say that $A$ \emph{satisfies the UCT}, if $(A,B)$ satisfies the UCT for every separable \ca\ $B$.
\end{df}

Phillips also showed that any Kirchberg algebra, satisfying the UCT or
not, has the same $K$-theory as a Kirchberg algebra that \emph{does} satisfy the UCT. In particular,
if there exist Kirchberg algebras that do not satisfy the UCT, then these algebras cannot be classified using the Elliott invariant.

Not being able to
prove that it holds automatically in the nuclear case\footnote{There are counterexamples for exact \ca s.}, the community began to regard the UCT as an additional condition that
ought to be added to the Elliott conjecture. The UCT assumption
remains until today as a rather mysterious one, and the problem of whether a nuclear \ca\ satisfies the UCT
is a major one, known as the \emph{UCT problem}.

\section{Classification through tracial approximations}
Despite the power of Elliott's intertwining argument, which was so successful in the classification of AF-algebras (and, more generally, AH-algebras of slow dimension growth), most of the natural examples of simple, nuclear \ca s
rarely come with canonical inductive limit decompositions using more tractable building blocks. It was thus
quickly realized that it was necessary to obtain classification results using assumptions that could be checked
in cases where the algebra is not given as a direct limit.

A major breakthrough in this direction was Lin's study and classification of C*-algebras of tracial rank zero.
Roughly speaking\footnote{This formulation is only correct whenever $A$ has strict comparison of positive elements by traces.},
a simple \uca\ $A$ is said to have \emph{tracial rank zero} if for every finite subset $F\subseteq A$ and for every $\ep>0$
there exist a projection $p\in A$, a finite dimensional \ca\ $C$, and a unital homomorphism $\varphi\colon C\to pAp$ such that,
with $p^\perp = 1-p$, one has
$\|a - pap - p^\perp a p^\perp\|<\ep$ for all $a\in F$ and $\tau(p^\perp)<\ep$ for all $\tau\in T(A)$.

Lin's result from \cite{Lin_classification_2004} reads as follows:

\begin{thm}\label{thm:TRzero}
Let $A$ and $B$ be separable, simple, unital, nuclear \ca s of tracial rank zero satisfying the UCT.
Then $A$ and $B$ are isomorphic if and only if $\mathrm{Ell}(A)\cong \mathrm{Ell}(B)$.
\end{thm}

One of the main advantages of the above result is that its main hypothesis, tracial rank zero, can be verified for a
wide variety of \ca s, including many arising from topological dynamical systems \cite{LinPhi_crossed_2010}, 
or even noncommutative dynamical
systems, without having to prove that the algebras in question admit tractable inductive limit decompositions.

Lin's tracial rank zero classification was shortly after extended to algebras of \emph{tracial rank one} \cite{Lin_simple_2007}, and further generalizations 

\section{Exotic examples: from Jiang-Su to Villadsen, R\o rdam and Toms}\label{sec:ExoticEgs}
The mid 1990s saw the construction of simple nuclear \ca s which were not believed (or hoped) to exist. These
examples would later revolutionize the Elliott programme, since some of these pathological examples
unequivocally demonstrated the need to modify Elliott's conjecture.

In his monograph \cite{Kap_aspects_1958}, Kaplansky asked whether every simple C*-algebra must contain a nontrivial
projection. The consensus at the time seemed to be that this would not be the case, and that the most likely counterexample
would be either the irrational rotation algebra $A_{\theta}$ or the reduced group C*-algebra $C^*_{\lambda}(\mathbb{F}_2)$.
The first candidate turned out to have lots of nontrivial projections (see \cite{Rie_associated_1981}), while the second one was indeed
proved to be simple \cite{Pow_simplicity_1975} and projectionless \cite{PimVoi_groups_1982}. It was then asked whether
a \emph{nuclear} simple \ca\ must contain a nontrivial projection. Even this question has a negative answer, as Blackadar
showed in \cite{Bla_simple_1980} and \cite{Bla_simple_1981}. A later construction due to Jiang and Su \cite{JiaSu_simple_1999}
attracted a great deal of attention due to its connections to Elliott's classification programme. Indeed, the algebra
$\mathcal{Z}$ constructed in \cite{JiaSu_simple_1999}, now known as the \emph{Jiang-Su algebra}, is an infinite-dimensional
simple, separable, unital, nuclear \ca\ satisfying the UCT with the same Elliott invariant as the complex numbers. Even
more, under a very mild condition on $K_0(A)$ called \emph{weak unperforation}\footnote{An ordered group $G$
is said to be \emph{weakly unperforated} if for any $g\in G$, one has $g\geq 0$ if and only if for some (for all) $n\in \{2,3,\ldots\}$
one has $ng\geq 0$.}, any \ca\ $A$ as in the assumptions of
\autoref{cnj:EC} satisfies $\mathrm{Ell}(A)\cong \mathrm{Ell}(A\otimes \mathcal{Z})$; see \cite{GonJiaSu_obstructions_2000}.

It follows that at the level of the Elliott invariant, and when $K_0$ is weakly unperforated,
the Jiang-Su algebra $\mathcal{Z}$ acts as a tensorial unit. Hence, if the Elliott Conjecture were to hold, every \ca\ with
weakly unperforated $K_0$-group should be isomorphic to its tensor product with $\mathcal{Z}$. These observations motivated
significant efforts to understand the structure of $\mathcal{Z}$-stable \ca s, which were regarded as the C*-algebraic
analogs of the McDuff factors. Among many other remarkable properties, $\mathcal{Z}$-stable \ca s enjoy a formidable dichotomy
relative to the structure of their projections: a simple $\mathcal{Z}$-stable \ca\ is either stably finite or purely infinite;
see \cite{GonJiaSu_obstructions_2000}. In particular, simple $\mathcal{Z}$-stable \ca s admit a type-like characterization
analogous to von Neumann factors.

The question of whether every infinite-dimensional \ca\ as in \autoref{cnj:EC} has weakly unperforated $K_0$-group was open at
the time, and it was settled shortly after by Villadsen \cite{Vil_simple_1998}. Villadsen's construction
of a simple, separable, nuclear
\ca\ whose $K_0$-group is not weakly unperforated answered a long-standing question of Blackadar: the order on projections
in this algebra is not determined by their behavior on traces. (This is in stark contrast
to the situation for II$_1$-factors.) The methods and constructions used by Villadsen
were later refined and extended to construct a number of ``pathological'' examples of nuclear \ca s, two which we proceed to discuss.

In \cite{Ror_simple_2003}, R\o rdam showed that there exists a simple, separable, nuclear \ca\ containing a finite and an
infinite projection; this \ca\ is in particular infinite but not purely infinite, and hence does not satisfy the dichotomy that
$\mathcal{Z}$-stable \ca s do. Perhaps most importantly, R\o rdam's example showed that Elliott's conjecture, at least in its original
form, was not true (even if the UCT is added as an assumption): the C*-algebra he constructed has the same Elliott
invariant as a purely infinite algebra (namely, its tensor product with $\mathcal{Z}$), yet it is not itself a
Kirchberg algebra.
Shortly after,
Toms \cite{Tom_classification_2008} constructed two \ca s $A$ and $B$ that could not be distinguished by virtually any
``reasonable'' functor\footnote{One for which it was conceivable that a range result could also be proved.}, including
the Elliott invariant, the real and stable ranks, and any homotopy-invariant continuous functor, yet were not
isomorphic. The construction of these algebras follows ideas of Villadsen and is simpler than related constructions:
the algebra $A$ is constructed as a direct limit of algebras of the form $M_k(C([0,1]^n))$, with connecting maps
defined using diagonal maps and point evaluations. (Of course, the choice of matrix sizes, dimension of the space, and
connecting maps must be done with extreme care.) The \ca\ $B$ is then the tensor product of $A$ with any UHF-algebra (and
it could also be taken to be $A\otimes\mathcal{Z}$).

A natural reaction to the incompleteness of the Elliott invariant is to expand it to include whatever data was used to
prove its incompleteness. This is, however, not always a good idea. For once, it is possible that the resulting invariant
is still incomplete, thus requiring further modifications. Even worse, one may incorporate so much new information that
a result proving its completeness may lack any practical impact. In this sense, not all counterexamples are of the same
quality.

While the examples of R\o rdam could be distinguished using their stable rank, Tom's examples required a much finer
invariant: the Cuntz semigroup. However, since range results
for Cuntz semigroups are still until today out of reach, a classification using this invariant would not be
accompanied with any result describing the range of the invariant (which all previous results did). Toms' counterexample
was therefore regarded as evidence \emph{not} of the fact that the invariant must be enlargened, but rather that the
class of \ca s one wishes to classify must be restricted.
The results from \cite{GonJiaSu_obstructions_2000} that were
discussed above strongly suggest that the class of separable, simple, nuclear, unital \ca s that are $\mathcal{Z}$-stable
may be particularly tractable, and that the Elliott invariant may not be the right invariant to use outside the
$\mathcal{Z}$-stable case. Further evidence of the fact that $\mathcal{Z}$-stable \ca s are distinctly tame was provided
by R\o rdam in \cite{Ror_stable_2004}, where it is shown that $\mathcal{Z}$-stable \ca s always have strict comparison of
positive elements by quasitraces\footnote{This means the following: for positive elements $a,b\in M_{\infty}(A)$,
one has $a\precsim b$ whenever $d_{\tau}(a)<d_{\tau}(b)$ for all 2-quasitraces $\tau$ on $A$, where
$d_{\tau}(a)=\lim\limits_{n\to\infty} \tau(a^{1/n})$.}; and that they have stable rank one whenever they are finite.

The possibility that $\mathcal{Z}$-stability had to be added to the assumptions in the Elliott
conjecture was quickly agreed upon, and this led to the \emph{revised Elliott conjecture}:

\begin{cnj}\label{cnj:RevEC}
Let $A$ and $B$ be simple, separable, unital, nuclear, $\mathcal{Z}$-stable \ca s. Then $A$ and $B$ are isomorphic if
and only if $\mathrm{Ell}(A)\cong \mathrm{Ell}(B)$.
\end{cnj}

In retrospect, the necessity of some form of tensorial absorption should not come as a surprise, since tensorial stability
is ubiqutous in results about classification and structure of operator algebra. For example, a major step in Connes'
groundbreaking work \cite{Con_classification_1976} involved showing that an injective II$_1$-factor absorbs the hyperfinite
II$_1$-factor tensorially; similarly, a highly praised result of Kirchberg \cite{KirPhi_embedding_2000} asserts that a
simple, nuclear, separable \ca\ is
purely infinite if and only if it absorbs the Cuntz algebra $\OI$ tensorially.

Whether the UCT has to be included as an assumption in the above conjecture is still up to debate. In particular, a
verification of \autoref{cnj:RevEC} would imply a positive solution to the UCT question. 

The fact that
a simple $\mathcal{Z}$-stable \ca\ is either stably finite or purely infinite, implies that \autoref{cnj:RevEC} has
two cases: the (purely) infinite and the (stably) finite.
Since the infinite case was beautifully settled
by the work of Kirchberg and Phillips, all efforts towards confirming the revised version of Elliott's conjecture
focused on the stably finite case\footnote{The results are more general than what we state here: nuclearity can be relaxed to exactness, and simplicity is not necessary.}. In this setting, a combination
of very deep results of Cuntz \cite{Cun_dimension_1978}, Blackadar-Handelman \cite{BlaHan_dimension_1982}, and
Haagerup \cite{Haa_quasitraces_2014}, implies that any stably finite, simple, nuclear, unital \ca\ has
a tracial state. This fact was of enormous significance since it opened the doors to the use of techniques from
II$_1$-factors in the study of stably finite, nuclear \ca s, via the GNS construction.

Assuming the UCT, Elliott's conjecture is now a theorem, arguably one of the greatest accomplishments of the field in the last couple of decades thanks to the contributions of dozens of mathematicians:

\begin{thm}
Let $A$ and $B$ be simple, separable, unital, nuclear, $\mathcal{Z}$-stable C*-algebras
satisfying the UCT. Then $A\cong B$ if and only if Ell$(A)\cong$ Ell$(B)$.
\end{thm}

\end{document}